\documentclass{amsart}

\usepackage{amsfonts,amsmath, soul, matlab-prettifier, bm, amsthm,enumitem,amssymb,multirow,float,mathtools,bbm,array,varwidth}
\usepackage[all]{xy}
\usepackage[margin=0.9in]{geometry}
\usepackage{graphicx}
\usepackage{mathtools}

\title{Computation of Hecke eigenvalues (mod $p$) via quaternions}
\author{Yiannis Fam}
\date{}

\theoremstyle{plain}
\newtheorem{thm}{Theorem}[section]

\newtheorem{prop}[thm]{Proposition}
\newtheorem{condition}[thm]{Condition}
\newtheorem{lemma}[thm]{Lemma}
\newtheorem{cor}[thm]{Corollary}
\newtheorem{algo}[thm]{Algorithm}
\theoremstyle{definition}
\newtheorem{defn}[thm]{Definition}
\newtheorem{notn}[thm]{Notation}
\newtheorem{rem}[thm]{Remark}
\newtheorem{example}[thm]{Example}

\newcommand{\norm}[1]{\left\lVert#1\right\rVert}
 
\begin{document}
	\maketitle
	\pagenumbering{arabic}

\begin{abstract}
In a 1987 letter \cite{Ser1}, Serre proves that the systems of Hecke eigenvalues arising from mod $p$ modular forms (of fixed level $\Gamma(N)$ coprime to $p$, and any weight $k$) are the same as those arising from functions $\Omega(N) \to \overline{\mathbb F}_p$, where $\Omega(N)$ is some double quotient of $D^\times (\mathbb A_f)$ and $D$ is the unique quaternion algebra over $\mathbb Q$ ramified at $\{p,\infty\}$. We present an algorithm which then computes these Hecke eigenvalues on the quaternion side in a combinatorial manner.
\end{abstract}

\section{Introduction}

The study of Hecke eigenvalues originated from Ramanujan's $\Delta$-function
$$\Delta(q) = q \prod\limits_{n=1}^\infty (1-q^n)^{24},$$
a weight 12 and level 1 Hecke eigenform about which Ramanujan made a number of conjectures that have motivated much of the theory of modular forms over the past century. More recently, it is known \cite{Del1} that a normalised Hecke eigenform $f=\sum a_n q^n$ in $S_k(\Gamma_0(N);\epsilon)$ gives rise to a mod $p$ Galois representation $\rho: \mathrm{Gal}(\overline{\mathbb Q}/\mathbb Q) \to \mathrm{GL}_2(\overline{\mathbb F}_p)$. This representation is unramified away from $pN$ and is characterised by the following trace and determinant of Frobenius data (mapped appropriately into $\overline{\mathbb F}_p$):
$$\mathrm{tr}(\rho(\mathrm{Frob}_\ell)) = a_\ell , \hspace{1cm} \mathrm{det}(\rho(\mathrm{Frob}_\ell )) = \ell^{k-1}\epsilon(\ell).$$
These representations play a central role in modern number theory, most notably in the conjectures of Serre. As such, the mod $p$ Hecke eigenvalues are objects of great interest, and so one may wish to enumerate these systems of Hecke eigenvalues as a source of examples.

At present, there already exist algorithms for computing these Hecke eigenvalues, for example using modular symbols in connection with the Eichler-Shimura theorem \cite{Wie1}. Our approach will instead make use of the following theorem of Serre (where we have restricted his result to a fixed level $\Gamma(N)$) \cite{Ser1}, that tells us that one could instead compute these Hecke eigenvalues by working with a particular quaternion algebra. The computation of the Hecke eigenvalues then becomes a combinatorial one, which is perhaps a more elementary approach.

\begin{thm}\label{Serre}
Let $D$ be the unique quaternion algebra over $\mathbb Q$ ramified at $\{p,\infty\}$. The systems of Hecke eigenvalues ($a_\ell$) (with $a_\ell \in \overline{\mathbb F}_p$, $\ell \nmid pN$ and fixed $N \geq 3$ coprime to $p$) coming from the modular forms (mod $p$) of level $\Gamma(N)$, are the same as those coming from the functions
$$\left(\Omega(N) := U(N) \backslash D^\times (\mathbb A_f) / D^\times (\mathbb Q)\right) \to \overline{\mathbb F}_p.$$
\end{thm}

We remark that in this theorem, we do not necessarily realise all the weight $k$ systems of eigenvalues from the modular form side as weight $(k \mod p^2-1)$ systems of eigenvalues on the quaternion side, only those arising from weight $k$ eigenforms not divisible by the Hasse invariant. They will still appear on the quaternion side when ranging over all weights.

\begin{notn}
We need to explain some of this notation:
\begin{itemize}
\item Let $\mathcal O$ be any maximal order of $D$.
\item Let $D_\ell = D \otimes \mathbb Q_\ell $ and $\mathcal O_\ell  = \mathcal O \otimes \mathbb Z_\ell $ a maximal order, for any prime $\ell$. So $\mathcal O_\ell  \cong \mathrm{M}_2(\mathbb Z_\ell )$ for $\ell\neq p$.
\item Let $D^\times(\mathbb A_f)$ denote the finite part of the adelic points of $D^\times$, in other words the restricted product ${\prod\limits_{\ell}} ' D^\times_\ell $ with respect to the subgroups $\mathcal O_\ell ^\times$.
\item Let $\pi \in \mathcal O$ be a uniformiser of $\mathcal O_p$.
\item Let $\mathcal O_p^\times(1)$ be the kernel of reduction mod $\pi$, $\mathcal O_p^\times \to \mathbb F_{p^2}^\times$.
\item For $\ell \neq p$, let $\mathcal O_\ell ^\times(N)$ be the subgroup of $\mathcal O_\ell ^\times \cong \mathrm{GL}_2(\mathbb Z_\ell )$ consisting of elements congruent to 1 mod $\ell^{v_\ell(N)}$, where $\ell^{v_\ell(n)}$ is the highest power of $\ell$ dividing $N$.
\item Let $U(N) = \mathcal O_p^\times(1) \times \prod\limits_{\ell\neq p} \mathcal O_\ell ^\times(N)$, an open subgroup of $D^\times (\mathbb A_f)$.
\item For $(x_\ell) \in D^{\times}(\mathbb A_f)$, we will denote by $[x_\ell]$ the image in $\Omega(N)$.
\item By a {\it weight $k$} function $f: \Omega(N) \to \overline{\mathbb F}_p$ we mean a function which satisfies $f(\mu \cdot [x_\ell ]) = \mu^{-k}f([x_\ell ])$, where $\mu \in \mathcal O_p^\times/ \mathcal O_p^\times(1) \cong \mathbb F_{p^2}^\times$ acts on $[x_\ell ]$ by multiplication in the $p$-place.
\end{itemize}
\end{notn}

Note that elements of $\Omega(N)$ correspond to isomorphism classes of invertible left $\mathcal O$-ideals $I$ with $\pi N$-structure, meaning a basis for $I/\pi N I$ as an $\mathcal O / \pi N \mathcal O$-module. Explicitly, an adelic point $(x_\ell )$ corresponds to an ideal $I$ with $I_\ell  = \mathcal O_\ell x_\ell $ for all $\ell$. The $\pi N$-structure is then given by the reduction modulo $\pi N I$ of any element $x \in I$ satisfying the congruences $$x \equiv x_\ell \mod \pi N I_\ell$$ for all $\ell$ (this congruence is vacuous for any $\ell \nmid pN$, so such $x$ exists by the Chinese remainder theorem). We then quotient $D^\times(\mathbb A_f)$ by $U(N)$ and $D^\times (\mathbb Q)$ exactly to get the desired bijection.

For a prime $\ell_0 \nmid pN$, the Hecke operator $T_{\ell_0  }$ on this space of functions $\Omega(N) \to \overline{\mathbb F}_p$ is given by $$T_{\ell_0  } f([x_\ell ]) = \ell_0  ^{-1}\sum\limits_i f(g_i \cdot [x_{\ell }])$$
for $\mathrm{GL}_2(\mathbb Z_{\ell_0  }) \begin{psmallmatrix} 1 & 0 \\ 0 & \ell_0   \end{psmallmatrix} \mathrm{GL}_2(\mathbb Z_{\ell_0  }) = \bigsqcup \mathrm{GL}_2(\mathbb Z_{\ell_0  })g_i$. Here $g_i \cdot [x_{\ell_0}]$ means we pick a representative $(x_\ell) \in D^\times (\mathbb A_f)$ of $[x_\ell]$, multiply this in the $\ell_0$-place by the matrix $g_i$ (under an identification of $\mathcal O_{\ell_0} \cong \mathrm{M}_2(\mathbb Z_{\ell_0})$ and hence of $D_{\ell_0} \cong \mathrm{M}_2(\mathbb Q_{\ell_0})$), and then take the image in $\Omega(N)$. Each individual $g_i \cdot [x_\ell]$ is not well defined in $[x_\ell]$, but $T_{l_0}$ is well defined, provided we pick the same representative $(x_\ell) \in D^\times (\mathbb A_f)$ for each multiplication. This Hecke module structure has been studied before, as in the likes of \cite{MR1846458}.

The algorithm we present computes a matrix for the Hecke operator $T_{\ell_0}$ on the space of (weight $k$) functions $\Omega(N) \to \overline{\mathbb F}_p$, where $\ell_0, p$ and $N$ are pairwise coprime. Note that we allow the cases $N=1,2$ as the definitions on the quaternion side still make sense here. We begin by computing an illustrative example for the case $p=11$, and then generalise this to weight $k=0$ and level $N=1$ for any $p$. Note that in this case the matrices we compute are exactly Brandt matrices. This is then extended to higher weight and level, essentially by keeping track of the $\pi N$-structure. 

We remark that a similar computation has been performed by Pizer in \cite{Piz1}. This has since been applied in, for example, \cite{MR2175121} and \cite{MR1863287}. Pizer was interested in computing the subspace of cusp forms on $\Gamma_0(N)$ generated by theta series, and the Hecke operators on this subspace. The algorithm involves computing certain Brandt matrix series. The main difference between our approach and Pizer's is that we incorporate level $N$ structure through coset representatives of $\mathcal O_\ell^\times (N) \backslash \mathcal O_\ell^\times \cong \mathrm{GL}_2(\mathbb Z / \ell^{v_\ell(N)} \mathbb Z)$. See also \cite{MR2193808} and \cite{MR2291849} for a similar approach in the context of Hilbert modular forms (the corresponding description of Hilbert modular forms on the quaternion side is more involved than for modular forms). Pizer instead works with orders of level $N$, which allows them to work purely with the quaternion algebras, without having to write down explicit isomorphisms of the form $\mathcal O_\ell \cong \mathrm{M}_2(\mathbb Z_\ell)$, something that we must do (see Section 3). On the other hand, the use of matrices perhaps makes the contribution from the level structure more visible. The tradeoff is then between computing these isomorphisms with the Hensel-like argument involved, and computing orders of level $N$. A related point of difference is that in Pizer's argument one must compute the left ideal classes of an order of level $N$, whereas we only need to compute these for our maximal order $\mathcal O$, because our level structure is already captured in $\mathcal O_\ell^\times (N) \backslash \mathcal O_\ell^\times \cong \mathrm{GL}_2(\mathbb Z / \ell^{v_\ell(N)} \mathbb Z)$, which is very concrete. Because a maximal order has minimal discriminant, the corresponding Minkowski bound is lower, and so the computation of these ideal classes in our case may be (somewhat) faster.

The author would like to thank his M.Phil. supervisors, Alex Ghitza and Chenyan Wu, for their careful reading of the script, and their valuable corrections and suggestions. Writing of this paper was partially supported by the University of Melbourne Robert George Williams Scholarship, as well as the Australian Government Research Training Program Scholarship.

\section{An Example}
\begin{example}
Take $p$ to be 11. The quaternion algebra $D$ is then $D = \left( \frac{-1, -11}{\mathbb Q}\right)$. In this example we will work with level $N=1$ as doing so greatly reduces the size of $\Omega(N)$. We will then compute the Hecke eigenvalues on the space of functions on $\Omega := U \backslash D^\times (\mathbb A_f) / D^\times (\mathbb Q)$, where $U := \mathcal O_p^\times \times \prod\limits_{l \neq p} \mathcal O_\ell ^\times$ is not quite $U(1)$ as we replace $\mathcal O_p^\times(1)$ with $\mathcal O_p^\times$. Note that this is the same as computing the Hecke eigenvalues on $\Omega(1)$ of weight $k \equiv 0 \mod p^2-1$. Indeed, the modularity condition tells us that these are functions on $\Omega(1)$ that are invariant under the action of $\mathcal O_p^\times$, and so can be identified with functions on $\Omega$. The computation in this section is based on notes of Buzzard \cite{Buz1}.

Firstly, let's understand the set $\Omega$. By the same argument that $\Omega(N)$ corresponds to (isomorphism classes of) invertible left $\mathcal O$-ideals $I$ with $\pi N$-structure, we see that $\Omega$ corresponds just to the invertible left $\mathcal O$-ideals $I$. Running the following in MAGMA \cite{MR1484478}, we compute a maximal order $\mathcal O$ of $D$, and then find that its left ideal class set has order 2, whose elements we call $I_1$ and $I_2$. So $\Omega$ has two elements. We can also compute integer bases for $\mathcal O$, $I_1$ and $I_2$.

\begin{verbatim}
>>	D := QuaternionAlgebra< RationalField() | -1, -11>; 
>>	O := MaximalOrder(D);
>>	Basis(O);
>>	Classes := LeftIdealClasses(O);
>>	#Classes;
>>	I1 := Classes[1];
>>	Basis(I1);
>>	I2 := Classes[2];
>>	Basis(I2);
\end{verbatim}

which outputs:

\begin{verbatim}
[ 1, i, 1/2*i + 1/2*k, 1/2 + 1/2*j ]
2
[ 1, -i, -1/2*i - 1/2*k, 1/2 - 1/2*j ]
[ 2, -2*i, 1 - 3/2*i - 1/2*k, 1/2 - i - 1/2*j ]
\end{verbatim}

We now want to rewrite $I_1$ and $I_2$ as adelic points in $\Omega$. For any invertible left $\mathcal O$-ideal $I$, $I \otimes \mathbb Z_\ell $ is a principal $\mathcal O_\ell$-ideal, generated by any nonzero element $\alpha_\ell $ whose reduced norm has minimal $\ell$-adic valuation. To see this, we refer to Corollary 16.6.12 in Voight \cite{Voi1} that any invertible semi-order (lattice that contains 1 and has reduced norm equal to the ring $R=\mathbb Z_\ell $) is an order, so that $(I \otimes \mathbb Z_\ell )\alpha_\ell ^{-1}$ is an invertible semi-order, which must then by $\mathcal O_\ell$. So if we are given a $\mathbb Z$-basis for $I$, then $I \otimes \mathbb Z_\ell $ is generated by a basis element whose reduced norm has minimal $\ell$-adic valuation. Computing the reduced norms of the given basis elements of $I_1$ and $I_2$ we get $[1,1,3,3]$ and $[4,4,6,4]$ respectively. So we see that $I_1 = \mathcal O$ (which was obvious anyway) and $I_2 \otimes \mathbb Z_\ell  = \mathcal O \otimes \mathbb Z_\ell $ for all $\ell \neq 2$. Moreover, $I_2 \otimes \mathbb Z_2 = \mathcal O \otimes \mathbb Z_2 \cdot (1-\frac{3}{2}i-\frac{1}{2}ij)$ as 6 has minimal 2-adic valuation. Thus $I_1$ corresponds to $w^1:=[1,1,\dots] \in \Omega$ and $I_2$ corresponds to $w^2:=[1-\frac{3}{2}i-\frac{1}{2}ij,1,1, \dots] \in \Omega$ (we use superscripts here to avoid overloading the subscripts, which we want to use for the places). An obvious choice of basis for the vector space of $\overline{\mathbb F}_p$ valued functions on $\Omega$ are the characteristic functions $\mathbbm{1}_{w^1}$ and $\mathbbm{1}_{w^2}$.

To compute the Hecke operator $T_{\ell_0  }$, we will need to work with matrices at the $\ell_0  $-place. Let's compute $T_2$ and $T_3$ with respect to the basis $\{\mathbbm{1}_{w^1},\mathbbm{1}_{w^2}\}$. We will need isomorphisms $\mathcal O \otimes \mathbb Z_2 \cong \mathrm{M}_2(\mathbb Z_2)$ and $\mathcal O \otimes \mathbb Z_3 \cong \mathrm{M}_2(\mathbb Z_3)$. In general, we have an isomorphism
\begin{align*}
D \otimes \mathbb Q_\ell  &\cong \mathrm{M}_2(\mathbb Q_\ell ) \\
i &\mapsto \begin{pmatrix} 0 & -1 \\ 1 & 0 \end{pmatrix}\\
j &\mapsto \begin{pmatrix} x & y \\ y & -x \end{pmatrix} 
\end{align*}
for $x,y \in \mathbb Q_\ell $ such that $x^2+y^2=-11$, when $\ell \neq 11$. For $\mathrm{M}_2(\mathbb Z_3)$, note that $\mathcal O \otimes \mathbb Z_3$ has $\mathbb Z_3$-basis $\{1,i,j,ij\}$, by using the basis for $\mathcal O$ computed above and noticing that 2 is invertible in $\mathbb Z_3$. So for any $x,y \in \mathbb Z_3$ with $x^2+y^2=-11$, the above isomorphism restricts to an injection $\mathcal O \otimes \mathbb Z_3 \hookrightarrow \mathrm{M}_2(\mathbb Z_3)$, which must in fact be an isomorphism by maximality. We could take $(x,y) = (\sqrt{-11},0)$ for $\sqrt{-11}$ a root of $x^2+11=0$ in $\mathbb Z_3$, choosing for example the root which is $1 \mod 3$ by Hensel's lemma.

For $\mathrm{M}_2(\mathbb Z_2)$ we need to be a little more careful because we need $1,i,\frac{1}{2}i + \frac{1}{2}ij, \frac{1}{2}+\frac{1}{2}j$ to all map to elements in $\mathrm{M}_2(\mathbb Z_2)$, rather than just $1,i,j,ij$. If we take $(x,y) = (\sqrt{-15},2)$ for $\sqrt{-15}$ a root of $x^2+15=0$ in $\mathbb Z_2$ (taking for example the root congruent to $1 \mod 4$ by Hensel's lemma), then we see that $\frac{1}{2} + \frac{1}{2}j$ maps to $\begin{psmallmatrix} \frac{1+\sqrt{-15}}{2} & 1 \\ 1 & \frac{1-\sqrt{-15}}{2}\end{psmallmatrix}$, which is in $\mathrm{M}_2(\mathbb Z_2)$. This gives us our desired isomorphism $\mathcal O \otimes \mathbb Z_2 \cong \mathrm{M}_2(\mathbb Z_2)$. Note that if we took instead $(x,y) = (2,\sqrt{-15})$, we would then map $\frac{1}{2}+\frac{1}{2}j$ to $\begin{psmallmatrix} \frac{3}{2} & \frac{\sqrt{-15}}{2} \\ \frac{\sqrt{-15}}{2} & -\frac{1}{2}\end{psmallmatrix}$, which is not in $\mathrm{M}_2(\mathbb Z_2)$.

We now begin our computation of $T_2$, starting with the value of $T_2(\mathbbm{1}_{w^1})(w^1)$. By definition, 
$$2T_2(\mathbbm{1}_{w^1})(w^1) = \mathbbm{1}_{w^1}\left(\begin{pmatrix} 1 & 0 \\ 0 & 2 \end{pmatrix} \cdot [1,1, \dots]\right) + \mathbbm{1}_{w^1}\left(\begin{pmatrix} 1 & 1 \\ 0 & 2 \end{pmatrix} \cdot [1,1, \dots]\right) + \mathbbm{1}_{w^1}\left(\begin{pmatrix} 2 & 0 \\ 0 & 1 \end{pmatrix} \cdot [1,1, \dots]\right)$$
so we reduce to checking, for example, whether $\begin{psmallmatrix} 1 & 0 \\ 0 & 2 \end{psmallmatrix} \cdot [1,1, \dots] = \left[\begin{psmallmatrix} 1 & 0 \\ 0 & 2 \end{psmallmatrix},1,1, \dots\right]$ is the same as $w^1=[1,1,\dots]$ or $w^2 = [1-\frac{3}{2}i-\frac{1}{2}ij,1,1, \dots]$ as an element of $\Omega$. 

The condition that $\left[\begin{psmallmatrix} 1 & 0 \\ 0 & 2 \end{psmallmatrix},1,1, \dots\right] = [1,1,\dots]$ in $\Omega$ is equivalent to saying that the $\mathcal O$-ideal $\mathcal J$ is principal, where $\mathcal J$ is defined by $\mathcal J \otimes \mathbb Z_2 = (\mathcal O \otimes \mathbb Z_2) \cdot \begin{psmallmatrix} 1 & 0 \\ 0 & 2 \end{psmallmatrix}$ and $\mathcal J \otimes \mathbb Z_\ell  = \mathcal O \otimes \mathbb Z_\ell $ for all $\ell \neq 2$. If this was the case, then we see that $\mathcal J$ must be generated by an element $\alpha$ of $\mathcal O$ (since $\mathcal J \otimes \mathbb Z_\ell  \subset \mathcal O \otimes \mathbb Z_\ell $ for all $\ell$) of reduced norm $2 = \mathrm{det}\begin{psmallmatrix} 1 & 0 \\ 0 & 2 \end{psmallmatrix}$. We have previously computed a basis $[1,i,\frac{1}{2}i+\frac{1}{2}ij,\frac{1}{2}+\frac{1}{2}j]$ for $\mathcal O$. Writing $\alpha = t+x \cdot i + y \cdot (\frac{1}{2}i+\frac{1}{2}ij) + z \cdot (\frac{1}{2}+\frac{1}{2}j)$, for $t,x,y,z \in \mathbb Z$, we compute
$$\mathrm{nrd}(\alpha) = (t+\frac{1}{2}z)^2 + (x + \frac{1}{2}y)^2 + \frac{11}{4}y^2 + \frac{11}{4}z^2.$$
For this to be equal to 2, we must have $y=z=0$ and $t^2=x^2=1$. So $\mathcal J$ must be $\mathcal O \cdot (1 \pm i)$. This satisfies $\mathcal J \otimes \mathbb Z_\ell  = \mathcal O \otimes \mathbb Z_\ell $ for $\ell \neq 2$, so we only need to check whether this works at 2. In other words, we need to check if
$$\mathrm{M}_2(\mathbb Z_2) \cdot \begin{pmatrix} 1 & 0 \\ 0 & 2 \end{pmatrix} = \mathrm{M}_2(\mathbb Z_2) \cdot \begin{pmatrix} 1& -1 \\ 1 & 1 \end{pmatrix}$$
or
$$\mathrm{M}_2(\mathbb Z_2) \cdot \begin{pmatrix} 1 & 0 \\ 0 & 2 \end{pmatrix} = \mathrm{M}_2(\mathbb Z_2) \cdot \begin{pmatrix} 1& 1 \\ -1 & 1 \end{pmatrix}$$
where our isomorphism $\mathcal O \otimes \mathbb Z_2 \cong \mathrm{M}_2(\mathbb Z_2)$ sends $1+i$ to $\begin{psmallmatrix} 1& -1 \\ 1 & 1 \end{psmallmatrix}$, and $1-i$ to $\begin{psmallmatrix} 1& 1 \\ -1 & 1 \end{psmallmatrix}$. Equivalently, we need to check if 
$$\begin{pmatrix} 1& -1 \\ 1 & 1 \end{pmatrix}\begin{pmatrix} 1 & 0 \\ 0 & 2 \end{pmatrix}^{-1} \in \mathrm{GL}_2(\mathbb Z_2) \hspace{1cm} \text{or} \hspace{1cm}
\begin{pmatrix} 1& 1 \\ -1 & 1 \end{pmatrix}\begin{pmatrix} 1 & 0 \\ 0 & 2 \end{pmatrix}^{-1} \in \mathrm{GL}_2(\mathbb Z_2).$$
We see that neither is true. So we have computed one of the terms in $T_2(\mathbbm{1}_{w^1})(w^1)$, namely
$$ \mathbbm{1}_{w^1}\left(\begin{pmatrix} 1 & 0 \\ 0 & 2 \end{pmatrix} \cdot [1,1, \dots]\right) = 0.$$

As a sanity check, we verify that $\left[\begin{psmallmatrix} 1 & 0 \\ 0 & 2 \end{psmallmatrix},1,1, \dots\right] = [1-\frac{3}{2}i-\frac{1}{2}ij,1,\dots] = w^2 \in \Omega$, so that $$ \mathbbm{1}_{w^2}\left(\begin{pmatrix} 1 & 0 \\ 0 & 2 \end{pmatrix} \cdot [1,1, \dots]\right) = 1$$
as we would expect. The isomorphism $\mathcal O \otimes \mathbb Z_2 \cong \mathrm{M}_2(\mathbb Z_2)$ sends $1-\frac{3}{2}i-\frac{1}{2}ij$ to $\begin{psmallmatrix} 2 & \frac{3-\sqrt{-15}}{2} \\ \frac{-3-\sqrt{-15}}{2} & 0 \end{psmallmatrix}$. Let $\mathcal J$ be the $\mathcal O$-ideal generated by the adelic point $\left(\begin{psmallmatrix} 1 & 0 \\ 0 & 2 \end{psmallmatrix},1,1, \dots\right)$, and $\mathcal I = I_2$ the ideal generated by $ (1-\frac{3}{2}i-\frac{1}{2}ij,1,\dots)$. Then the condition that they correspond to the same element of $\Omega$ is equivalent to the existence of some $\alpha \in D^\times (\mathbb Q)$ such that $\mathcal J=\mathcal I\alpha$. Checking this locally, this means $\alpha \in (\mathcal O \otimes \mathbb Z_\ell )^\times$ for $\ell \neq 2$, and at $\ell =2$ we have 
\begin{equation}\label{alpha} \begin{pmatrix} 2 & \frac{3-\sqrt{-15}}{2} \\ \frac{-3-\sqrt{-15}}{2} & 0 \end{pmatrix} \alpha \begin{pmatrix} 1 & 0 \\ 0 & 2 \end{pmatrix}^{-1} \in \mathrm{GL}_2(\mathbb Z_2). \end{equation}
We can rewrite this as $$\alpha \in \begin{pmatrix} 2 & \frac{3-\sqrt{-15}}{2} \\ \frac{-3-\sqrt{-15}}{2} & 0 \end{pmatrix}^{-1} \mathrm{GL}_2(\mathbb Z_2)  \begin{pmatrix} 1 & 0 \\ 0 & 2 \end{pmatrix}$$ so we see that $2\alpha \in \mathrm{M}_2(\mathbb Z_2)$, by inverting the matrix and clearing the denominator. It follows that if $\alpha$ exists, we must have $2\alpha \in \mathcal O$, and $\mathrm{nrd}(\alpha)=1$, by checking locally. There are finitely many possibilities for $2\alpha$, namely $\pm 2, \pm 2i, \pm i \pm \frac{1}{2} \pm \frac{1}{2}j, \pm 1 \pm \frac{1}{2}i \pm \frac{1}{2}ij$. We need to check if any of these satisfies (\ref{alpha}) when we replace $\alpha$ with the corresponding matrix in $\frac{1}{2}\mathrm{M}_2(\mathbb Z_2)$. A computation shows that one can take $\alpha = \frac{1}{2}( i + \frac{1}{2}(1+j))$, where we need that actually $\sqrt{-15} \equiv 1 \mod 8$ for our choice of square root.

Returning to our computation of $T_2(\mathbbm{1}_{w^1})(w^1)$, we need to compute the remaining terms $ \mathbbm{1}_{w^1}\left(\begin{psmallmatrix} 1 & 1 \\ 0 & 2 \end{psmallmatrix} \cdot [1,1, \dots]\right)$ and $ \mathbbm{1}_{w^1}\left(\begin{psmallmatrix} 2 & 0 \\ 0 & 1 \end{psmallmatrix} \cdot [1,1, \dots]\right)$. By the same argument as above, it suffices to check whether
$$\begin{pmatrix} 1& -1 \\ 1 & 1 \end{pmatrix}\begin{pmatrix} 1 & 1 \\ 0 & 2 \end{pmatrix}^{-1} \in \mathrm{GL}_2(\mathbb Z_2) \hspace{1cm} \text{or} \hspace{1cm}
\begin{pmatrix} 1& 1 \\ -1 & 1 \end{pmatrix}\begin{pmatrix} 1 & 1 \\ 0 & 2 \end{pmatrix}^{-1} \in \mathrm{GL}_2(\mathbb Z_2)$$
for the first term, and whether
$$\begin{pmatrix} 1& -1 \\ 1 & 1 \end{pmatrix}\begin{pmatrix} 2 & 0 \\ 0 & 1 \end{pmatrix}^{-1} \in \mathrm{GL}_2(\mathbb Z_2) \hspace{1cm} \text{or} \hspace{1cm}
\begin{pmatrix} 1& 1 \\ -1 & 1 \end{pmatrix}\begin{pmatrix} 2 & 0 \\ 0 & 1 \end{pmatrix}^{-1} \in \mathrm{GL}_2(\mathbb Z_2)$$
for the second. The first two are true and the last two are not. So we deduce that 
$$\mathbbm{1}_{w^1}\left(\begin{pmatrix} 1 & 1 \\ 0 & 2 \end{pmatrix} \cdot [1,1, \dots]\right) = 1 \hspace{1cm} \text{and} \hspace{1cm} \mathbbm{1}_{w^1}\left(\begin{pmatrix} 2 & 0 \\ 0 & 1 \end{pmatrix} \cdot [1,1, \dots]\right) = 0.$$
Therefore $$2T_2(\mathbbm{1}_{w^1})(w^1) = 0+1+0=1.$$
We can also deduce that $2T_2(\mathbbm{1}_{w^2})(w^1) = 3-1=2$.

One similarly shows that $2T_2(\mathbbm{1}_{w^1})(w^2) = 3$ and so $2T_2(\mathbbm{1}_{w^2})(w^2)=0$. The required computation is to show that the $\mathcal O$-ideals generated by $\left( g \cdot \begin{psmallmatrix} 2 & \frac{3-\sqrt{-15}}{2} \\ \frac{-3-\sqrt{-15}}{2} & 0 \end{psmallmatrix}, 1, 1, \dots\right)$, for $g = \begin{psmallmatrix} 1 & 0 \\ 0 & 2 \end{psmallmatrix}, \begin{psmallmatrix} 1 & 1 \\ 0 & 2 \end{psmallmatrix}$ and $ \begin{psmallmatrix} 2 & 0 \\ 0 & 1 \end{psmallmatrix}$, are principal $\mathcal O$-ideals. The generators can be taken to be $\alpha = -i+\frac{1}{2} - \frac{1}{2}j, i + \frac{1}{2}-\frac{1}{2}j$ and $2$ respectively.

It follows that $2T_2(\mathbbm{1}_{w^1}) = \mathbbm{1}_{w^1} + 3\mathbbm{1}_{w^2}$ and $2T_2(\mathbbm{1}_{w^2}) = 2\mathbbm{1}_{w^1}$, so $T_2$ has matrix $\frac{1}{2}\begin{psmallmatrix} 1 & 2 \\ 3 & 0 \end{psmallmatrix}$ with respect to the basis $\{\mathbbm{1}_{w^1},\mathbbm{1}_{w^2}\}$. The eigenvalues are $\lambda = \frac{-2}{2},\frac{3}{2} \in \overline{\mathbb F}_{11}$.

The Hecke operator $T_3$ can be computed in a similar way. We give one example evaluating the summand $$ \mathbbm{1}_{w^2}\left(\begin{pmatrix} 1 & 0 \\ 0 & 3 \end{pmatrix} \cdot \left[\begin{pmatrix} 2 & \frac{3-\sqrt{-15}}{2} \\ \frac{-3-\sqrt{-15}}{2} & 0 \end{pmatrix},1,1, \dots \right]\right) =  \mathbbm{1}_{w^2}\left( \left[\begin{pmatrix} 2 & \frac{3-\sqrt{-15}}{2} \\ \frac{-3-\sqrt{-15}}{2} & 0 \end{pmatrix},\begin{pmatrix} 1 & 0 \\ 0 & 3 \end{pmatrix},1, 1,\dots \right]\right)$$ of $3T_3(\mathbbm{1}_{w^2})(w^2)$. One could do this by checking whether the $\mathcal O$-ideal $\mathcal J$ generated by the adelic point in the square brackets is principal. If it is, then this value of $\mathbbm{1}_{w^2}$ is $0$. This only works because $\Omega$ has two elements. In general, we need to check if $\mathcal J$ is in the same left ideal class as the $\mathcal O$-ideal $\mathcal I = I_2$ generated by the adelic point $\left( \begin{psmallmatrix} 2 & \frac{3-\sqrt{-15}}{2} \\ \frac{-3-\sqrt{-15}}{2} & 0 \end{psmallmatrix}, 1, 1, \dots\right)$. So we need to check if there exists $\alpha \in D^\times$ such that $\mathcal J = \mathcal I \alpha$. Checking this locally, we need (for the two isomorphisms $D \otimes \mathbb Q_2 \cong \mathrm{M}_2(\mathbb Q_2)$ and $D \otimes \mathbb Q_3 \cong \mathrm{M}_2(\mathbb Q_3)$ specified)
\begin{equation}\label{alpha2}\begin{pmatrix} 2 & \frac{3-\sqrt{-15}}{2} \\ \frac{-3-\sqrt{-15}}{2} & 0 \end{pmatrix} \alpha \begin{pmatrix} 2 & \frac{3-\sqrt{-15}}{2} \\ \frac{-3-\sqrt{-15}}{2} & 0 \end{pmatrix}^{-1} \in \mathrm{GL}_2(\mathbb Z_2) \end{equation}
and
\begin{equation}\label{alpha3}\alpha \begin{pmatrix} 1 & 0 \\ 0 & 3 \end{pmatrix}^{-1} \in \mathrm{GL}_2(\mathbb Z_3)\end{equation}
and $\alpha \in (\mathcal O \otimes \mathbb Z_\ell )^\times$ for $\ell \neq 2,3$. The first two conditions imply that $2\alpha \in \mathcal O \otimes \mathbb Z_2$ and $\alpha \in \mathcal O \otimes \mathbb Z_3$. So $2\alpha \in \mathcal O$, and $\mathrm{nrd}(\alpha) = 3$. There are finitely many possibilities for $2\alpha$.

Before we go through all the possibilities for $2\alpha$ and check whether they satisfy the two equations, we remark that for equation (\ref{alpha2}) we only really need to work with $2\alpha$ and the matrix $\begin{psmallmatrix} 2 & \frac{3-\sqrt{-15}}{2} \\ \frac{-3-\sqrt{-15}}{2} & 0 \end{psmallmatrix}$ modulo some power of 2 (and for equation (\ref{alpha3}) we only need $2\alpha$ modulo some power of 3). Indeed, using the fact that our choice of $\sqrt{-15}$ is congruent to $1 \mod 8$, we have 
$$\begin{pmatrix} 2 & \frac{3-\sqrt{-15}}{2} \\ \frac{-3-\sqrt{-15}}{2} & 0 \end{pmatrix} \equiv \begin{pmatrix} 2 & 1 \\ 2 & 0 \end{pmatrix} \mod 4.$$
Suppose $2\alpha$ has matrix $\begin{psmallmatrix} a_2 & b_2 \\ c_2 & d_2 \end{psmallmatrix} \in \mathrm{M}_2(\mathbb Z_2)$. Then equation (\ref{alpha2}) tells us (after inverting the matrix) that we need 
$$\begin{pmatrix} 2 & \frac{3-\sqrt{-15}}{2} \\ \frac{-3-\sqrt{-15}}{2} & 0 \end{pmatrix}\begin{pmatrix} a_2 & b_2 \\ c_2 & d_2 \end{pmatrix}  \begin{pmatrix} 0 & \frac{-3+\sqrt{-15}}{2} \\ \frac{3+\sqrt{-15}}{2} & 2 \end{pmatrix} \in 4\mathrm{M}_2(\mathbb Z_2).$$
Note that the norm condition on $\alpha$ will then guarantee that we land in $4 \mathrm{GL}_2(\mathbb Z_2)$ not just $4 \mathrm{M}_2(\mathbb Z_2)$. Working modulo 4, we need
$$\begin{pmatrix} 2 & 1 \\ 2 & 0 \end{pmatrix} \begin{pmatrix} a_2 & b_2 \\ c_2 & d_2 \end{pmatrix}  \begin{pmatrix} 0& -1 \\ 2 & 2 \end{pmatrix} \equiv 0 \mod 4.$$
Expanding this out, this is equivalent to 
\begin{equation}\label{condition2} 2 \mid a_2, \hspace{3em} 4 \mid c_2, \hspace{3em} 2 \mid d_2. \end{equation}
Similarly, if $2\alpha$ has matrix $\begin{psmallmatrix} a_3 & b_3 \\ c_3 & d_3 \end{psmallmatrix} \in \mathrm{M}_2(\mathbb Z_3)$, then equation (\ref{alpha3}) is equivalent to 
\begin{equation}\label{condition3} 3 \mid b_3, \hspace{3em} 3 \mid d_3. \end{equation}

So now we look through $2\alpha \in \mathcal O$ of reduced norm 12. Write 
$$2\alpha = t\cdot 1 + x \cdot i + y \cdot \frac{1}{2}(i+ij) + z \cdot \frac{1}{2}(1+j)$$ for $t,x,y,z \in \mathbb Z$. We also compute the images of the basis elements $1,i, \frac{1}{2}(i+ij), \frac{1}{2}(1+j)$ in $\mathrm{M}_2(\mathbb Z_2)$ and $\mathrm{M}_2(\mathbb Z_3)$, then reduce them modulo 4 and 3 respectively:
$$ i  \mapsto \begin{pmatrix} 0 & -1 \\ 1 & 0 \end{pmatrix} \mod 4 \hspace{7em}  i  \mapsto \begin{pmatrix} 0 & -1 \\ 1 & 0 \end{pmatrix} \mod 3$$
$$ \frac{1}{2}(i+ij)  \mapsto \begin{pmatrix} -1 & 0 \\ 1 & 1 \end{pmatrix} \mod 4 \hspace{7em}  \frac{1}{2}(i+ij)  \mapsto \begin{pmatrix} 0 & 0 \\ 1 & 0 \end{pmatrix} \mod 3.$$
$$ \frac{1}{2}(1+j)  \mapsto \begin{pmatrix} 1 & 1 \\ 1 & 0 \end{pmatrix} \mod 4 \hspace{7em}  \frac{1}{2}(1+j)  \mapsto \begin{pmatrix} 1 & 0 \\ 0 & 0 \end{pmatrix} \mod 3$$
It follows that
$$ 2\alpha  \mapsto \begin{pmatrix} t-y+z & -x+z \\ x+y+z & t+y \end{pmatrix} \mod 4 \hspace{7em}  2\alpha  \mapsto \begin{pmatrix} t+z & -x \\ x+y & t \end{pmatrix} \mod 3.$$
Hence equations (\ref{condition2}) and (\ref{condition3}) tell us that there exists $2\alpha \in \mathcal O$ of reduced norm 12 satisfying equations (\ref{alpha2}) and (\ref{alpha3}), if and only if we can find $t,x,y,z \in \mathbb Z$ such that
\begin{equation*} 2 \mid t-y+z, t+y, \hspace{2em} 4 \mid x+y+z, \hspace{2em} 3 \mid x,t, \hspace{2em} (t+\frac{1}{2}z)^2 + (x + \frac{1}{2}y)^2 + \frac{11}{4}y^2 + \frac{11}{4}z^2 = 12\end{equation*}
where the last equation tells us the reduced norm is 12. We can check that there are no such solutions. For example, we see that $z$ must be even, and then from the norm condition we see that $\lvert z \rvert \leq 2$. Trying $z = 2$, the norm condition tells us that we must have $(t,x,y) = (0,0,0), (-2,0,0)$ or $(-1,\pm1,0)$, none of which satisfy the congruence conditions. For $z=-2$ we similarly must have $(t,x,y) = (0,0,0), (2,0,0)$ or $(1,\pm1,0)$ which also does not work. And finally for $z=0$ we need $t^2 + (x+\frac{1}{2}y)^2 + \frac{11}{4}y^2 = 12$. Since $3 \mid t$, in one case we have $t^2=9$, and then in $(x+\frac{1}{2}y)^2+\frac{11}{4}y^2 = 3$ there are no solutions with $4 \mid x+y$ and $3 \mid x$. Otherwise we have $t=0$ and $(x+\frac{1}{2}y)^2+\frac{11}{4}y^2 = 12$ with $2 \mid y$, $4 \mid x+y$ and $3 \mid x$, which has no solutions, checking $y = 0, \pm 2$.

We deduce that  $$ \mathbbm{1}_{w^2}\left(\begin{pmatrix} 1 & 0 \\ 0 & 3 \end{pmatrix} \cdot \left[\begin{pmatrix} 2 & \frac{3-\sqrt{-15}}{2} \\ \frac{-3-\sqrt{-15}}{2} & 0 \end{pmatrix},1,1, \dots \right]\right) = 0.$$
Performing the remaining calculations, using the relevant mod 3 and 4 matrices we have already written down above, we find that $T_3$ has matrix $\frac{1}{3}\begin{psmallmatrix} 2 & 2 \\ 3 & 1 \end{psmallmatrix}$ with respect to the basis $\{\mathbbm{1}_{w^1},\mathbbm{1}_{w^2}\}$. We compute the eigenvalues to be $\frac{-1}{3}, \frac{4}{3} \in \overline{\mathbb F}_{11}$.

Comparing the matrices for $T_2$ and $T_3$, they are simultaneously diagonalisable (as expected) with eigenvectors $\scriptsize\begin{pmatrix} 2 \\ -3 \end{pmatrix}$ and $\scriptsize\begin{pmatrix} 1 \\ 1 \end{pmatrix}$. The corresponding eigenvalues are $(a_2,a_3)=\left(\frac{-2}{2}, \frac{-1}{3}\right) = (-1,-4)$ and $(a_2,a_3)=\left(\frac{3}{2},\frac{4}{3}\right)=(7,5)$, viewing $a_2,a_3$ as elements of $\overline{\mathbb F}_{11}$. In view of Serre's Theorem \ref{Serre}, one might wonder what mod 11 modular forms give rise to these eigenvalues. The first comes from $\theta^9(\Delta) = q-q^2-4q^3+\dots \mod 11$, which has minimal weight filtration 120 (meaning that 120 is the lowest weight at which a mod 11 modular form has this $q$-expansion), and the second from the Hasse invariant $E_{10}$.

\end{example}

\section{Explicit isomorphisms $\mathcal O \otimes \mathbb Z_\ell  \cong \mathrm{M}_2(\mathbb Z_\ell )$}

We now work with any prime $p$. Let $D$ be the unique quaternion algebra over $\mathbb Q$ ramified exactly at $\{p,\infty\}$. It turns out that one can take $D \cong \left( \frac{-1,-p}{\mathbb Q}\right)$ and $\left( \frac{-2,-p}{\mathbb Q}\right)$ when respectively, $p \equiv 3 \mod 4$ (or $p=2$), and $p \equiv 5 \mod 8$. In the remaining case $p \equiv 1 \mod 8$, we can take $D \cong  \left( \frac{-r,-p}{\mathbb Q}\right)$ for any prime $r \equiv 3 \mod 4$ with $\left(\frac{r}{p}\right) = -1$. This can be verified by computations with the Hilbert symbol as in \cite{Voi1} Chapter 12. We will write $D = \left(\frac{-\epsilon,-p}{\mathbb Q}\right)$ for an appropriate $\epsilon$.

In the example of the previous section, we wrote down explicit isomorphisms $\mathcal O \otimes \mathbb Z_\ell  \cong \mathrm{M}_2(\mathbb Z_\ell )$. In general this is tricky to do, complicated by the fact that $\mathcal O$ might properly contain the order with $\mathbb Z$-basis $\{1,i,j,ij\}$ - we can have nontrivial denominators. See also \cite{MR3156561} for an algorithm for computing such isomorphisms, which arose as a byproduct of other more involved algorithms. Let $\mathcal O$ have $\mathbb Z$-basis $\{s^1,s^2,s^3,s^4\}$. What we need to do is find matrices $A,B \in \mathrm{M}_2(\mathbb Z_\ell )$ (corresponding to $i,j$) such that $$A^2 = -\epsilon, \hspace{2em} B^2 = -p, \hspace{2em}AB=-BA,$$ and such that the matrices corresponding to $s^1,s^2,s^3,s^4$, which a priori have entries in $\mathbb Q_\ell $, actually have entries in $\mathbb Z_\ell $. This latter condition just imposes some congruences on the entries of $A$ and $B$. For example, in the case $p=11$ using the given basis for $\mathcal O$, we also require that 
$$\frac{1}{2}(A+AB) \in \mathrm{M}_2(\mathbb Z_\ell ) \hspace{5em} \text{and} \hspace{5em} \frac{1}{2}(1+B) \in \mathrm{M}_2(\mathbb Z_\ell )$$
which only imposes congruence conditions when $\ell =2$. If we can do this in the general case, then this gives us an isomorphism $D \otimes \mathbb Q_\ell  \cong \mathrm{M}_2(\mathbb Q_\ell )$, which restricts to an injection $\mathcal O \otimes \mathbb Z_\ell  \hookrightarrow \mathrm{M}_2(\mathbb Z_\ell )$. But $\mathcal O \otimes \mathbb Z_\ell $ is a maximal order in $D \otimes \mathbb Q_\ell $, so this injection must actually be our desired isomorphism.

The main observation is that, for our purposes, we really only need the matrices corresponding to the $s^1,s^2,s^3,s^4$ modulo some sufficiently large power of $\ell$. Consequently, we only need to compute $A$ and $B$ modulo some $\ell^{n_\ell }$. This was hinted at in the computation of $T_3$ in the example above. We will describe a formula for a sufficiently large $n_\ell $ later. So we need to search for $A_0$ and $B_0$ in $\mathrm{M}_2(\mathbb Z / \ell^{n_\ell }\mathbb Z)$ for some $n_\ell $, which satisfy $$A_0^2 = -\epsilon, \hspace{2em}B_0^2 = -p, \hspace{2em}A_0B_0=-B_0A_0 \mod \ell^{n_\ell }$$ and the congruence conditions imposed by the basis of $\mathcal O$ (we take $n_\ell $ sufficiently large so that these congruences can be viewed as congruences modulo $\ell^{n_\ell }$). We know that such $A_0$ and $B_0$ exist because we know that there exists an isomorphism $\mathcal O \otimes \mathbb Z_\ell  \cong \mathrm{M}_2(\mathbb Z_\ell )$. Since $i, j \in \mathcal O$ have reduced trace 0, we can furthermore assume that $A_0$ and $B_0$ are trace-free. Then such a solution can be found by a finite enumeration of all matrices in $\mathrm{M}_2(\mathbb Z/\ell^{n_\ell } \mathbb Z)$ (a more efficient method would be to use the lemma below repeatedly). The claim then is that these solutions can be lifted to our desired $A$ and $B$ (although we do not need to write down $A$ and $B$, just know that $A_0$ and $B_0$ lift). Note that $A$ and $B$ will automatically satisfy the congruence conditions from the basis of $\mathcal O$ because we are lifting from $A_0$ and $B_0$. This is the content of the following lemma (for $\ell \neq 2$). 

\begin{lemma}\label{lift matrices}
Suppose we have trace-free matrices $A_0,B_0 \in \mathrm{M}_2(\mathbb Z / \ell^m \mathbb Z)$ for $\ell \neq p$ an odd prime and $m \geq 2$, and suppose they satisfy $$A_0^2 = -\epsilon \mod \ell^m,\hspace{2em} B_0^2 = -p  \mod \ell^m, \hspace{2em}A_0B_0=-B_0A_0 \mod \ell^{m}.$$ Then we can find trace-free matrices $A_1,B_1 \in \mathrm{M}_2(\mathbb Z / \ell^{m+1} \mathbb Z)$ satisfying $$A_1^2 = -\epsilon \mod \ell^{m+1},\hspace{2em} B_1^2 = -p \mod \ell^{m+1}, \hspace{2em}A_1B_1=-B_1A_1 \mod \ell^{m+1}$$ with $A_1 \equiv A_0 \mod \ell^m$ and $B_1 \equiv B_0 \mod \ell^m$. It follows that we can lift to desired matrices $A,B \in \mathrm{M}_2(\mathbb Z_\ell )$ by induction.
\end{lemma}
\begin{proof}
Our proof involves writing out all the matrix entries and multiplying them together. The requirement that the matrices be trace-free simplify our calculations. Let $$A_0 = \begin{pmatrix} a_1 & a_2 \\ a_3 & -a_1 \end{pmatrix} \mod \ell^m \hspace{5em} B_0 = \begin{pmatrix} b_1 & b_2 \\ b_3 & -b_1 \end{pmatrix} \mod \ell^m.$$ Lifting these to (trace-free) matrices modulo $\ell^{m+1}$, we will write $$A_1 = A_0 + \ell^m X \mod \ell^{m+1} \hspace{5em} B_1 = B_0 + \ell^m Y \mod \ell^{m+1}$$ for trace-free matrices $$X = \begin{pmatrix} x_1 & x_2 \\ x_3 & -x_1 \end{pmatrix} \mod \ell \hspace{5em} Y = \begin{pmatrix} y_1 & y_2 \\ y_3 & -y_1 \end{pmatrix} \mod \ell.$$

A useful calculation is the following: if $P = \begin{psmallmatrix} p_1 & p_2 \\ p_3 & -p_1 \end{psmallmatrix}$ and $Q = \begin{psmallmatrix} q_1 & q_2 \\ q_3 & -q_1 \end{psmallmatrix}$ are trace-free matrices, then 
$$PQ+QP = \begin{pmatrix} 2p_1q_1 + p_2q_3 + p_3q_2 & 0 \\ 0 & 2p_1q_1+p_2q_3+p_3q_2\end{pmatrix}.$$
It follows that if we lift $A_0$ and $B_0$ to {\it any} trace-free matrices modulo $\ell^{m+1}$, also denoted $A_0$ and $B_0$, we have that 
$$A_0^2 = -\epsilon + \ell^m (*) \mod \ell^{m+1}, \hspace{2em} B_0^2 = -p + \ell^m (*) \mod \ell^{m+1}, \hspace{2em} A_0B_0+B_0A_0 = \ell^m(*) \mod \ell^{m+1}$$ where the $(*)$ are all {\it scalar} matrices modulo $\ell$. The conditions on $X$ and $Y$ coming from those on $A_1$ and $B_1$ are the following:
$$2a_1x_1+a_2x_3+a_3x_2 = (*) \mod \ell, \hspace{2em} 2b_1y_1+b_2y_3+b_3y_2 = (*) \mod \ell $$ $$ (2a_1y_1+a_2y_3+a_3y_2) +( 2b_1x_1+b_2x_3+b_3x_2) = (*) \mod \ell$$ for some scalars $(*) \mod \ell$. We can rewrite this as 
$$\begin{pmatrix}
 2a_1 & a_3 & a_2 & 0 & 0 & 0 \\
0 & 0 & 0 & 2b_1 & b_3 & b_2 \\
2b_1 & b_3 & b_2 & 2a_1 & a_3 & a_2 
 \end{pmatrix} \cdot \small\begin{pmatrix} x_1 \\ x_2 \\ x_3 \\ y_1 \\ y_2 \\ y_3 \end{pmatrix} = (*) \mod \ell$$
for some arbitrary vector $(*) \in \mathbb F_\ell ^3$. This is always possible if the vectors $\begin{pmatrix} 2a_1 & a_3 & a_2\end{pmatrix} \mod \ell$ and $\begin{pmatrix} 2b_1 & b_3 & b_2 \end{pmatrix} \mod \ell$ are linearly independent. Since $\ell \neq 2$, linear dependence is equivalent to the existence of a nonzero scalar $\lambda$ such that $A_0 = \lambda B_0 \mod \ell$, or that either $A_0$ or $B_0$ is 0 mod $\ell$. The latter two cases are not possible because otherwise $A_0^2 = 0 \mod \ell^2$ or $B_0^2 = 0 \mod \ell^2$. Since $m \geq 2$, this means that $\ell^2$ divides $\epsilon$ or $p$, which does not happen (we had to use $\ell^2$ in case $\epsilon = r$ and $\ell =r$). And in the first case, $A_0B_0+B_0A_0=0 \mod \ell$ tells us that $A_0^2=B_0^2=0 \mod \ell$, which is false since $\ell \neq p$. Hence we can always lift $A_0,B_0 \mod \ell^m$ to $A_1,B_1 \mod \ell^{m+1}$.
\end{proof}

\subsection{The case $\ell =2$}

When $\ell =2$, the argument of Lemma \ref{lift matrices} does not work, analogous to the difficulty of using a naive Hensel's lemma for finding square roots in $\mathbb Z_2$. We will make use of a generalised Hensel's lemma for multiple variables, which can be found as Theorem 3.3 in Conrad's notes \cite{Con1}, specialised to the case of $\mathbb Q_2$ with its usual absolute value $\lvert \cdot \rvert$.

\begin{thm}\label{Hensel}
Let $\mathbf{f} = (f_1,f_2,\dots, f_d) \in \mathbb Z_2[X_1,X_2,\dots,X_d]^d$ and $\mathbf{a}=(a_1,\dots,a_d) \in \mathbb Z_2^d$ satisfy 
$$\norm{\mathbf{f}(\mathbf{a})} < \lvert J_{\mathbf{f}}(\mathbf{a}) \rvert^2$$
where $J_{\mathbf{f}}$ is the Jacobian of $f$ - the determinant of its derivative matrix - and the norm of a vector is defined to be the maximum of the absolute values of its entries. Then there is a unique $\mathbf{\alpha} \in \mathbb Z_2^d$ such that $\mathbf{f}(\mathbf{\alpha})=0$ and $\norm{\mathbf{\alpha}-\mathbf{a}} < \lvert J_{\mathbf{f}}(\mathbf{a})\rvert$.
\end{thm}

To see how to apply this, we are looking for matrices $A,B \in \mathrm{M}_2(\mathbb Z_2)$ satisfying $$A^2 = -\epsilon, \hspace{2em} B^2 = -p, \hspace{2em}AB=-BA,$$ which also satisfy certain congruence conditions mod 4 (the highest power of 2 possibly dividing denominators in $s^1,s^2,s^3,s^4$ - see Proposition \ref{pizer basis} to follow). Because we know $\mathcal O \otimes \mathbb Z_2 \cong \mathrm{M}_2(\mathbb Z_2)$, we know that such $A$ and $B$ exist, so we can find such a solution modulo some power of 2 bigger than 4 and try to lift to $\mathbb Z_2$. If $A=\begin{psmallmatrix} a_1 & a_2 \\ a_3 & -a_1 \end{psmallmatrix}$ and $B=\begin{psmallmatrix} b_1 & b_2 \\ b_3 & -b_1 \end{psmallmatrix}$ with entries in $\mathbb Z_2$, then the conditions $$A^2 = -\epsilon, \hspace{2em} B^2 = -p, \hspace{2em}AB=-BA,$$
are equivalent to
\begin{equation}\label{defining}
a_1^2 + a_2a_3 = -\epsilon, \hspace{3em} b_1^2+b_2b_3 = -p, \hspace{3em} 2a_1b_1+a_2b_3+a_3b_2 = 0.
\end{equation}

\begin{lemma}
Suppose we have $a,b,c,x,y,z \in \mathbb Z$ satisfying
\begin{equation}\label{defining2} a^2+bc=-\epsilon \mod 128, \hspace{3em} x^2+yz = -p \mod 128, \hspace{3em} 2ax+bz+cy=0 \mod 128. \end{equation}
Then these are congruent to some $a_1,a_2,a_3,b_1,b_2,b_3 \in \mathbb Z_2$ modulo 128 respectively, satisfying equation (\ref{defining}).
\end{lemma}
\begin{proof}
We have 3 equations and 6 variables, so to apply Theorem \ref{Hensel} we need to fix 3 variables. For example, if we fixed $(b,c,x)$ and considered the polynomials 
$$f_1(X_1,X_2,X_3) = X_1^2+bc+\epsilon, \hspace{2em} f_2(X_1,X_2,X_3) = x^2+X_2X_3 + p, \hspace{2em} f_3(X_1,X_2,X_3) = 2xX_1+bX_3+cX_2$$
then we know that $\mathbf{f}(a,y,z) = 0 \mod 128$. Note $128 = 2^7$. The derivative matrix is $$(D\mathbf{f})(X) = \begin{pmatrix} 2X_1 & 0 & 0 \\ 0 & X_3 & X_2 \\ 2x & c & b \end{pmatrix}.$$
In order to lift to a solution  $a_1,a_2,a_3,b_1,b_2,b_3 \in \mathbb Z_2$, we need $16=2^4$ to not divide 
$$J_{\mathbf{f}}(a,y,z) = 2a(bz-cy).$$
Whether this holds depends on the $a,b,c,x,y,z$ we were given. If this does not hold for our particular $a,b,c,x,y,z$, then we could instead fix some of the other variables, replacing $\mathbf{f}$ with some $\mathbf{g}$, and check whether 16 divides the new value of $J_{\mathbf{g}}$. This means that the proof reduces to an analysis of several cases depending on the parities of the $a,b,c,x,y,z$.

\begin{notn}
Let $\mathbf{f}_{r,s,t}$ denote the length 3 vector of polynomials in 3 variables obtained by fixing the variables \textbf{other than} $r,s,t \in \{a,b,c,x,y,z\}$. So for example we considered the case $(r,s,t)=(a,y,z)$ above.
\end{notn}
We compute the following Jacobians:
\begin{equation}\label{Jac1}
(D \mathbf{f}_{a,x,b})(X) = 
\begin{pmatrix}
2X_1 & 0 & c\\
0 & 2X_2 & 0 \\
2X_2 & 2X_1 & z
\end{pmatrix}
\hspace{3em}
J_{\mathbf{f}_{a,x,b}}(a,x,b) = 4x(az-cx).
\end{equation}

\begin{equation}\label{Jac2}
(D \mathbf{f}_{a,x,c})(X) = 
\begin{pmatrix}
2X_1 & 0 & b\\
0 & 2X_2 & 0 \\
2X_2 & 2X_1 & y
\end{pmatrix}
\hspace{3em}
J_{\mathbf{f}_{a,x,c}}(a,x,c) = 4x(ay-bx).
\end{equation}

\begin{equation}\label{Jac3}
(D \mathbf{f}_{a,x,y})(X) = 
\begin{pmatrix}
2X_1 & 0 & 0\\
0 & 2X_2 & z \\
2X_2 & 2X_1 & c
\end{pmatrix}
\hspace{3em}
J_{\mathbf{f}_{a,x,y}}(a,x,y) = 4a(cx-az).
\end{equation}

\begin{equation}\label{Jac4}
(D \mathbf{f}_{a,x,z})(X) = 
\begin{pmatrix}
2X_1 & 0 & 0\\
0 & 2X_2 & y \\
2X_2 & 2X_1 & b
\end{pmatrix}
\hspace{3em}
J_{\mathbf{f}_{a,x,z}}(a,x,z) = 4a(bx-ay).
\end{equation}

\begin{equation}\label{Jac5}
(D \mathbf{f}_{b,y,z})(X) = 
\begin{pmatrix}
c & 0 & 0\\
0 & X_3 & X_2 \\
X_3 & c & X_1
\end{pmatrix}
\hspace{3em}
J_{\mathbf{f}_{b,y,z}}(b,y,z) = c(bz-cy).
\end{equation}

Now we split into several cases. Firstly, assume $p,\epsilon \neq 2$ are odd. We will frequently use from equation (\ref{defining2}) that $bz+cy$ is divisible by $2ax$ mod 128, and in particular is even.

\begin{enumerate}[label=(\roman*)]
\item \textbf{$a,x$ both even.} From equation (\ref{defining2}), we see that $b,c,y,z$ are all odd. But $4 \mid bz+cy$ and so $4 \nmid bz-cy$. So using (\ref{Jac5}), we can lift using the fact that $16 \nmid c(bz-cy)$.
\item \textbf{$a$ odd and $x$ even.} We see that $y,z$ are odd and $bc$ is even. Then use (\ref{Jac3}) and the fact that $16 \nmid 4a(cx-az)$.
\item \textbf{$a$ even and $x$ odd.} We see that $b,c$ are odd and $yz$ is even. Then use (\ref{Jac1}) and the fact that \\ $16 \nmid 4x(az-cx)$.
\item \textbf{$a,x$ both odd.} Then $a^2 \equiv x^2 \equiv 1 \mod 4$. From the way we defined $\epsilon$, we have that $p \equiv -\epsilon \mod 4$ (when neither is equal to 2). By symmetry, suppose $p \equiv 1 \mod 4$. Then we see that we must have $yz \equiv 2 \mod 4$. By symmetry, suppose $y$ is odd and $z$ is even but not divisible by 4. Then since $bz+cy$ is even, we must have that $c$ is even. If we then have $b$ even, use (\ref{Jac4}), and $16 \nmid 4a(bx-ay)$. If instead $b$ is odd, then because $4 \nmid z$ we know $bz \equiv 2 \mod 4$. But $a$ and $x$ are odd, so $4 \nmid 2ax$ and therefore $4 \nmid bz+cy$. It follows that $c$ is divisible by 4. Then we use (\ref{Jac3}), where $4a(cx-az)$ is divisible by 8 but not by 16. The other cases are symmetric.
\end{enumerate}

Finally we consider the case when either $p$ or $\epsilon$ is 2 (we cannot have both). Firstly, consider $\epsilon = 2$, so $p \equiv 5 \mod 8$.

\begin{enumerate}[label=(\roman*)]
\item \textbf{$a,x$ both even.} We see that $y,z$ are odd. Since $a^2+bc=-2 \mod 128$ and $a$ is even, exactly one of $b,c$ is even. But then $2 \nmid bz+cy$ is a contradiction.
\item \textbf{$a$ odd and $x$ even.} We see that $b,c,y,z$ are all odd. We also have from $2ax \mid bz+cy$ that $4 \mid bz+cy$, and therefore $4 \nmid bz-cy$. Then use (\ref{Jac5}) and the fact that $16 \nmid c(bz-cy)$.
\item \textbf{$a$ even and $x$ odd.} We see that $bc$ and $yz$ are even. Because $a^2+bc=-2 \mod 128$ and $a$ is even, exactly one of $b,c$ is even. Similarly, $x^2+yz \equiv -1 \mod 4$ and $x$ is odd, so exactly one of $y,z$ is even. We also know that $bz+cy$ is even. Thus either $b,y$ are even and $c,z$ are odd, or $b,y$ are odd and $c,z$ are even. In the first case use (\ref{Jac1}) and in the second case use (\ref{Jac2}).
\item \textbf{$a,x$ both odd.} We see that $b,c$ are odd and $yz$ is even. Again we have that exactly one of $y,z$ is even. If $y$ is even use (\ref{Jac2}) and if $z$ is even use (\ref{Jac1}).
\end{enumerate}

The only property of $p$ that we use is that $p \equiv 1 \mod 4$. If we took $p=2$, then $\epsilon =1$ is also 1 mod 4. So this case is symmetric to the above.
\end{proof}

The upshot of all this work is the following:

\begin{cor}
Let $\ell$ be any prime, and let $n_\ell $ be any integer at least 2, where we also ask that $n_2 \geq 7$. Suppose we can find trace-free matrices $A_0, B_0 \in \mathrm{M}_2(\mathbb Z/ \ell^{n_\ell } \mathbb Z)$ satisfying
$$A_0^2 = -\epsilon,\hspace{2em} B_0^2 = -p, \hspace{2em}A_0B_0=-B_0A_0 \mod \ell^{n_\ell }.$$
Then these lift to matrices $A,B \in \mathrm{M}_2(\mathbb Z_\ell )$ with $A \equiv A_0 \mod \ell^{n_\ell }$ and $B \equiv B_0 \mod \ell^{n_\ell }$ such that 
$$A^2 = -\epsilon,\hspace{2em} B^2 = -p, \hspace{2em}AB=-BA .$$
\end{cor}

So suppose we can find for $n_\ell  \geq 2$, or $n_2 \geq 7$ (for example by exhaustion of finitely many matrices in $\mathrm{M}_2(\mathbb Z/\ell^{n_\ell }\mathbb Z)$, or working through the Hensel argument for $\ell \neq 2$) such matrices $A_0$ and $B_0$ which satisfy the congruence conditions determined by the basis of $\mathcal O$. In other words, if we mapped $i \mapsto A_0$ and $j \mapsto B_0$, then the induced map on $s^1,s^2,s^3,s^4$ sends them to well defined matrices modulo $\ell^{n_\ell -2}$ (the 2 accounts for denominators when writing the $s^1,s^2,s^3,s^4$ in terms of $i,j$ - see below). Then $A_0$ and $B_0$ lift to $A,B \in \mathrm{M}_2(\mathbb Z_\ell )$ such that the map determined by $i \mapsto A$ and $j \mapsto B$ gives an isomorphism $\mathcal O \otimes \mathbb Z_\ell  \cong \mathrm{M}_2(\mathbb Z_\ell )$.

\begin{rem}
The following Proposition in Pizer's paper \cite{Piz1} (Proposition 5.2) gives explicit $\mathbb Z$-bases for a maximal order of our quaternion algebra $D$. We see in particular that the only primes dividing a denominator are 2 and $r$ (when $p \equiv 1 \mod 8$ and $\epsilon = r$), with power at most $2^2$ and $r$.

\begin{prop}\label{pizer basis}
A maximal order of $D$ is given by the $\mathbb Z$-basis:
\begin{align*}
&\frac{1}{2}(1+i+j+ij),i,j,ij  &\text{ if } p = 2 \\
&\frac{1}{2}(1+j), \frac{1}{2}(i+ij),j,ij &\text{ if } p \equiv 3 \mod 4 \\
&\frac{1}{2}(1+j+ij), \frac{1}{4}(i+2j+ij),j,ij &\text{ if } p \equiv 5 \mod 8 \\
&\frac{1}{2}(1+j), \frac{1}{2}(i+ij), \frac{1}{r}(j+a \cdot ij),ij &\text{ if } p \equiv 1 \mod 8
\end{align*}
where $a$ is some integer such that $r \mid a^2p+1$.
\end{prop}
This means we can explicitly write down the congruences we want our matrices $A_0$ and $B_0$ in the above Corollary to satisfy. We know that these matrices exist for $\ell \neq p$ by the ramification properties of the quaternion algebra $D$.
\begin{condition}\label{hensel condition}
We want to find matrices $A_0, B_0 \in \mathrm{M}_2(\mathbb Z / \ell^{n_\ell} \mathbb Z)$, for $n_\ell \geq 2$ and $n_2 \geq 7$ satisfying 
$$A_0^2 = -\epsilon,\hspace{2em} B_0^2 = -p, \hspace{2em}A_0B_0=-B_0A_0 \mod \ell^{n_\ell },$$
and also
\begin{align*}
&2 \mid 1+A_0+B_0+A_0B_0			&\text{ if } p = 2 \\
&2 \mid 1+B_0, \hspace{1em} 2 \mid A_0+A_0B_0				&\text{ if } p \equiv 3 \mod 4 \\
&2 \mid 1+B_0+A_0B_0, \hspace{1em} 4 \mid A_0+2B_0+A_0B_0				 &\text{ if } p \equiv 5 \mod 8 \\
&2 \mid 1+B_0, \hspace{1em} 2 \mid A_0+A_0B_0, \hspace{1em} r \mid B_0+a\cdot A_0B_0				&\text{ if } p \equiv 1 \mod 8
\end{align*}
where we remind ourselves that when $p \equiv 1 \mod 8$, we need $\epsilon = r$ is a prime congruent to $3 \mod 4$ with $\left(\frac{r}{p}\right)=-1$, and $a$ is some integer such that $r \mid a^2p+1$, which we choose in writing down the basis in Proposition \ref{pizer basis}. Note that these last four congruences are vacuous unless $\ell = 2$ or $r$.
\end{condition}
\end{rem}

\section{Weight 0 mod $p^2 - 1$ and level 1, for general $p$}

We now work with any prime $p$. Let $D$ be the unique quaternion algebra over $\mathbb Q$ ramified exactly at $\{p,\infty\}$, $\mathcal O$ a maximal order, and let $\Omega := U \backslash D^\times (\mathbb A_f) / D^\times (\mathbb Q)$, where $U := \mathcal O_p^\times \times \prod\limits_{l \neq p} \mathcal O_\ell ^\times$. We will compute the Hecke operator $T_{\ell_0  }$ on $\Omega$, which will give us the weight $0 \mod p^2-1$ eigenvalues on $\Omega(1)$. The argument in the example above largely generalises to this case. Recall the definition of $T_{\ell_0}$.

\begin{defn}
For a prime $\ell_0 \neq p$, the Hecke operator $T_{\ell_0  }$ on the space of functions $\Omega \to \overline{\mathbb F}_p$ is given by $$T_{\ell_0  } f([x_\ell ]) = \ell_0  ^{-1}\sum\limits_i f(g_i \cdot [x_{\ell  }])$$
for $\mathrm{GL}_2(\mathbb Z_{\ell_0  }) \begin{psmallmatrix} 1 & 0 \\ 0 & \ell_0   \end{psmallmatrix} \mathrm{GL}_2(\mathbb Z_{\ell_0  }) = \bigsqcup \mathrm{GL}_2(\mathbb Z_{\ell_0  })g_i$. Recall that $g_i \cdot [x_{\ell_0}]$ means we pick a representative $(x_\ell) \in D^\times (\mathbb A_f)$ of $[x_\ell]$, multiply this in the $\ell_0$-place by the matrix $g_i$ (under an identification of $\mathcal O_{\ell_0} \cong \mathrm{M}_2(\mathbb Z_{\ell_0})$ and hence of $D_{\ell_0} \cong \mathrm{M}_2(\mathbb Q_{\ell_0})$), and then take the image in $\Omega$.
\end{defn}

We explain part of this using the matrices of the previous section. Suppose we had some $w^i=[w_2^i, w_3^i, \dots] \in \Omega$ and $w^j=[w_2^j,w_3^j,\dots] \in \Omega$ corresponding to left ideal classes of $\mathcal O$ (so almost all $w_\ell^i,w_\ell^j$ can be taken to be 1), and we wanted to know if $g_k \cdot w^j = w^i \in \Omega$, for some $g_k$ appearing in $T_{\ell_0}$. Denote by $\mathcal J$ the left $\mathcal O$-ideal with local generators $\{w_2^j, w_3^j, \dots, g_k \cdot w_{\ell_0  }^j, \dots\}$, where we view $g_k \in \mathrm{M}_2(\mathbb Z_{\ell_0  })$ as an element of $\mathcal O_{\ell_0}$. Denote by $\mathcal I$ the left $\mathcal O$-ideal with local generators $\{w_2^i, w_3^i, \dots\}$. Then we need to check whether there exists $\alpha \in D^\times (\mathbb Q)$ such that $\mathcal J = \mathcal I \alpha$ (the question of determining when two $\mathcal O$-ideals are in the same ideal class, for $\mathcal O$ an Eichler order, has been studied in \cite{MR2967473}, but in our case we have local generators for the ideals, meaning that a basis for an ideal is not obvious). This means that for all $\ell \neq \ell_0  $, we require 
\begin{equation}\label{maineq1}w_\ell ^i \alpha (w_\ell ^j)^{-1} \in \mathcal O_\ell^\times\end{equation}
and also
\begin{equation}\label{maineq2}w_{\ell_0  }^i \alpha (w_{\ell_0  }^j)^{-1} g_k^{-1} \in \mathcal O_{\ell_0}^\times.\end{equation}
We deduce the following conditions on $\alpha$: let $\mathcal V_{i,j}$ be the set of primes $\ell$ at which at least one of $w_\ell ^i$ and $w_\ell ^j$ is not 1, throwing out $p$ and $\ell_0  $. For all $\ell$, we temporarily define $m_\ell = \mathrm{max}(\mathrm{nrd}(w_\ell^i), \mathrm{nrd}(w_\ell^j))$. Then for all $\ell \not\in \mathcal V_{i,j} \cup \{p,\ell_0  \}$, equation (\ref{maineq1}) is equivalent to 
\begin{equation}\label{cond1}\alpha \in \mathcal O_\ell^\times \hspace{2em} \ell \not\in \mathcal V_{i,j} \cup \{p,\ell_0  \}. \end{equation}
For $\ell =p$, because $D$ is ramified at $p$, equation (\ref{maineq1}) is equivalent to 
\begin{equation}\label{cond2}v_p(\mathrm{nrd}(\alpha)) = v_p(\mathrm{nrd}(w_\ell ^j))-v_p(\mathrm{nrd}(w_\ell ^i)),\end{equation}
where $v_p$ denotes the usual $p$-adic valuation on $\mathbb Z$. For $\ell \in \mathcal V_{i,j}$, equation (\ref{maineq1}) tells us that 
\begin{equation}\label{cond3}v_\ell (\mathrm{nrd}(\alpha)) = v_\ell (\mathrm{nrd}(w_\ell ^j))-v_\ell (\mathrm{nrd}(w_\ell ^i)) \hspace{2em} \ell \in \mathcal V_{i,j}.\end{equation}
Additionally, since $\alpha \in (w_\ell ^i)^{-1} \cdot \mathcal O_\ell^\times \cdot w_\ell ^j$, and because $\mathrm{nrd}(w_\ell ^i)$ has $\ell$-adic valuation at most $m_\ell $ by definition, we see that 
\begin{equation}\label{cond4} \ell^{m_\ell } \alpha \in \mathcal O_\ell  \cong \mathrm{M}_2(\mathbb Z_\ell ) \hspace{2em} \ell \in \mathcal V_{i,j}.\end{equation}
A similar argument using equation (\ref{maineq2}) shows that 
\begin{equation}\label{cond5}v_{\ell_0  }(\mathrm{nrd}(\alpha)) = v_{\ell_0  }(\mathrm{nrd}(w_{\ell_0  }^j))-v_\ell (\mathrm{nrd}(w_{\ell_0  }^i)) + 1\end{equation}
and
\begin{equation}\label{cond6} \ell_0  ^{m_{\ell_0  }+1} \alpha \in \mathcal O_{\ell_0} \cong \mathrm{M}_2(\mathbb Z_{\ell_0  }).\end{equation}

Combining equations (\ref{cond1}) to (\ref{cond6}), we see that we can compute rational numbers $$M = \ell_0  \prod\limits_{\ell \in \mathcal V_{i,j} \cup \{p,\ell_0\}} \ell^{m_\ell }$$ and $$K = \ell_0  \prod\limits_{\ell \in \mathcal V_{i,j} \cup \{p,\ell_0\}} \ell^{v_\ell (\mathrm{nrd}(w_\ell ^j)) - v_\ell (\mathrm{nrd}(w_\ell ^i))}$$ such that 
\begin{equation*}\label{Cond} M\alpha \in \mathcal O \hspace{2em} \text{and} \hspace{2em} \mathrm{nrd}(\alpha) = K.\end{equation*}

Note that we can rewrite, for $\ell \neq p, \ell_0  $, the condition $w_\ell ^i \alpha (w_\ell ^j)^{-1} \in \mathcal O_\ell$ (being a unit is then guaranteed by the norm condition) as saying 
$$w_\ell ^i \cdot M\alpha \cdot \overline{w_\ell ^j} \in \ell^{m_\ell  + v_\ell (\mathrm{nrd}(w_\ell ^j))} \mathcal O_\ell ,$$
where $\overline{w_\ell ^j}$ denotes the standard involution in $D \otimes \mathbb Q_\ell $. We can do the same for $\ell_0$, using $\overline{g_k} = \mathrm{adj}(g_k)$ the adjugate matrix of $g_k$, which gives an extra factor of $\ell_0$. Since $m_\ell  + v_\ell (\mathrm{nrd}(w_\ell ^j)) \leq 2m_\ell $, to check this condition we only really need $w_\ell ^i$, $M\alpha$ and $w_\ell ^j$ modulo $\ell^{2m_\ell }$. Similarly, we only really need $w_{\ell_0  }^i$, $M\alpha$ and $w_{\ell_0  }^j$ modulo $\ell_0  ^{2m_{\ell_0  }+2}$. This means we can make use of the matrices computed in the previous section to rephrase our conditions on $\alpha$. Let $s^1,s^2,s^3,s^4$ be an integer basis for $\mathcal O$, taking for instance the basis of Proposition \ref{pizer basis}. Since $M\alpha \in \mathcal O$, we can write 
$$M\alpha = t\cdot s^1+x\cdot s^2 + y\cdot s^3 + z \cdot s^4$$ for variables $t,x,y,z$ which are to take values in $\mathbb Z$. Let $\ell$ be a prime in $\mathcal V_{i,j} \cup \{\ell_0  \}$. We can identify $M\alpha \mod \ell^{2m_\ell +2}$ with a matrix $\begin{psmallmatrix} a_\ell & b_\ell  \\ c_\ell  &-a_\ell \end{psmallmatrix} \in \mathrm{M}_2(\mathbb Z/ \ell^{2m_\ell +2} \mathbb Z)$, where $a_\ell ,b_\ell ,c_\ell $ are $(\mathbb Z/ \ell^{2m_\ell +2} \mathbb Z)$-linear functions in $t,x,y,z$ which can be computed. We can also, for any $\ell \neq p$, identify $w_\ell^i$ and $w_\ell^j$ with matrices $W_\ell^i,W_\ell^j \in \mathrm{M}_2(\mathbb Z/ \ell^{2m_\ell +2}\mathbb Z)$.

 We have that $g_k \cdot w^j = w^i \in \Omega$ if and only if we can find some $t,x,y,z \in \mathbb Z$ such that the following conditions hold:
\begin{equation}\label{maincond}
\begin{dcases*}
\mathrm{nrd}(M\alpha) = KM^2 \\
W_\ell ^i \cdot \begin{pmatrix} a_\ell & b_\ell  \\ c_\ell  &-a_\ell \end{pmatrix} \cdot\mathrm{adj}(W_\ell ^j) \in \ell^{m_\ell  + v_\ell (\det(W_\ell ^j))} \cdot \mathrm{M}_2(\mathbb Z/ \ell^{2m_\ell +2} \mathbb Z) \hspace{3em} \text{ for $\ell \in \mathcal V_{i,j}$} \\
W_{\ell_0  }^i \cdot \begin{pmatrix} a_{\ell_0  }& b_{\ell_0  } \\ c_{\ell_0  } &-a_{\ell_0  }\end{pmatrix} \cdot \mathrm{adj}(W_{\ell_0  }^j) \cdot \mathrm{adj}(g_k)
 \in \ell_0  ^{m_{\ell_0  } + v_{\ell_0  }(\det(W_{\ell_0  }^j))+2} \cdot \mathrm{M}_2(\mathbb Z/ \ell_0  ^{2m_{\ell_0} +2} \mathbb Z).
\end{dcases*}
\end{equation}

The first condition tells us that $t,x,y,z \in \mathbb Z$ satisfy some quadratic equation. Because the reduced norm is a positive definite quadratic form in the $t,x,y,z$ (due to ramification at $\infty$), this means that there are only finitely many solutions to the quadratic equation. We can then enumerate them (for example, diagonalising the quadratic form, computing the finitely many solutions with the new basis, and then solving for $t,x,y,z$). Once we do so, it remains to check whether they satisfy the last two conditions of (\ref{maincond}). By expanding them out, these conditions can be interpreted as congruence conditions on $t,x,y,z$ modulo $\ell^{2m_\ell +2}$ for $\ell \in \mathcal V_{i,j} \cup \{\ell_0  \}$, or more precisely modulo $ \ell^{m_\ell  + v_\ell (\det(W_\ell ^j))}$ and $\ell_0  ^{m_{\ell_0  } + v_{\ell_0  }(\det(W_{\ell_0  }^j))+2}$. This allows us to determine whether $g_k \cdot w^j = w^i \in \Omega$, and hence compute a matrix representative for the Hecke operator $T_{\ell_0}$. We are now ready to present the algorithm. This should be read in conjunction with the following Table of Notation. The column for 'Corresponding Matrices' refers to matrices generated using the methods of Section 3.

\newpage
\begin{center}
\textbf{Table of Notation}
\end{center}
\begin{center}

\binoppenalty=9999
\relpenalty = 9999
	\begin{tabular}{|V{16em} |V{23em}|V{9em}|}
	\hline
\textbf{Notation} & \textbf{Definition} & \textbf{Corresponding Matrices} \\ \hline
$D = \left(\frac{-\epsilon,-p}{\mathbb Q}\right)$ & Quaternion algebra ramified at $\{p,\infty\}$. & \\ \hline
$\mathcal O$ & A maximal order of $D$. & \\ \hline
$I_1,\dots, I_h$ & Representatives of the left ideal classes of $\mathcal O$. & \\ \hline
$\mathcal B_1,\dots, \mathcal B_h$ & $\mathbb Z$-bases for $I_1,\dots,I_h$. & \\ \hline
$\mathcal V$ & The set of primes $\ell$ for which in some $\mathcal B_j$ all elements have reduced norm divisible by $\ell$, excluding $p$ and including $\ell_0  $. These are the primes at which we need to compute matrices. & \\ \hline
$i,j$ & Generators for $D$. & $A_\ell ,B_\ell  \in \mathrm{M}_2(\mathbb Z/\ell^{n_\ell }\mathbb Z)$ for each $\ell \in \mathcal V$.  \\ \hline
$\{s^1,s^2,s^3,s^4\}$ & A $\mathbb Z$-basis for $\mathcal O$. & $S^1_\ell ,S_\ell ^2,S_\ell ^3,S_\ell ^4 \in \mathrm{M}_2(\mathbb Z/\ell^{2m_\ell +2} \mathbb Z)$ for each $\ell \in \mathcal V$.  \\ \hline
$w^j=[w_2^j,w_3^j,w_5^j, \dots] \in \Omega$ & The elements of $\Omega$ corresponding to the $I_j$. The square brackets means the double coset represented by the adelic point $(w_2^j,w_3^j,\dots)$. From the way we compute this, almost all $w_\ell ^j$ will be 1. & $W_\ell ^j \in \mathrm{M}_2(\mathbb Z/ \ell^{2m_\ell +2} \mathbb Z)$ corresponding to the $w_\ell ^j$ for each $\ell \in \mathcal V$.  \\ \hline
$m_\ell $, $n_\ell =2m_\ell +4$ for $\ell \neq 2$, and $n_2 = \max(7,2m_2 + 4)$ & $m_\ell := \max\limits_j(v_\ell(\mathrm{nrd}(w_\ell^j)))$. & \\ \hline
$g_0,\dots,g_{\ell_0  } \in \mathrm{M}_2(\mathbb Z_{\ell_0  })$ & The matrices  $\begin{psmallmatrix} 1 & 0 \\ 0 & \ell_0   \end{psmallmatrix}, \begin{psmallmatrix} 1 & 1 \\ 0 & \ell_0   \end{psmallmatrix}, \dots, $ $\begin{psmallmatrix} 1 & \ell_0  -1 \\ 0 & \ell_0   \end{psmallmatrix}, \begin{psmallmatrix} \ell_0   & 0 \\ 0 & 1 \end{psmallmatrix}$. & \\ \hline
$\mathbbm{1}_{w^j}$ & The characteristic function $\Omega \to \overline{\mathbb F}_p$ of the point $w^j \in \Omega$. & \\ \hline
$e_{i,j,k} \in \{0,1\}$ & $e_{i,j,k} = \mathbbm{1}_{w^i}(g_k \cdot w^j)$ for $1 \leq i,j \leq h$ and $0 \leq k \leq \ell_0  $. The multiplication $g_k \cdot w^j$ means $[w_2^j,w_3^j, \dots, g_k \cdot w_{\ell_0} ^j, \dots]$, with multiplication occuring only in the $\ell_0$-th place. & \\ \hline
$\mathcal V_{i,j} \subset \mathcal V$ & The set of primes $\ell$ such that at least one of $w_\ell ^i$ and $w_\ell ^j$ is not 1, excluding $p$ \textbf{and} $\ell_0  $. & \\ \hline
	\end{tabular}
\end{center}

\newpage

\begin{algo}\label{algo}
Input: distinct primes $p$ and $\ell_0$.

Output: a matrix representing the action of the Hecke operator $T_{\ell_0}$ on the space of all functions $\Omega \to \overline{\mathbb F}_p$.
\begin{enumerate}[label=(\roman*)]
\item Define a quaternion algebra $D = \left(\frac{-\epsilon,-p}{\mathbb Q}\right)$ over $\mathbb Q$, ramified exactly at $\{p,\infty\}$. Define a maximal order $\mathcal O$ with integer basis given as in Proposition \ref{pizer basis}, for which we denote the basis elements $\{s^1,s^2,s^3,s^4\}$. Compute the left ideal classes $I_1, \dots, I_h$ of $\mathcal O$, and bases $\mathcal B_1, \dots, \mathcal B_h$ for them.

\item\label{step0} Compute the points $w^j = [w_2^j, w_3^j, w_5^j, \dots] \in \Omega$ corresponding to $I_j$ for each $j$ as follows. We take $w_\ell ^j$ to be any generator of $I_j \otimes \mathbb Z_\ell $ (which we know is principal). To do this, for our basis $\mathcal B_j$, compute the reduced norm of each of the four elements and set $w_\ell ^j$ to be any of these elements whose reduced norm has minimal $\ell$-adic valuation. Note that for almost all $\ell$ this valuation is zero, so we can instead take $w_\ell ^j=1$, and do so when possible.

\item\label{step1} Determine the set $\mathcal V$, defined in the Table of Notation. For each $\ell \in \mathcal V \cup \{p\}$, compute $m_\ell$ and $n_\ell$. For each $\ell \in \mathcal V$, compute matrices $A_\ell, B_\ell \in \mathrm{M}_2(\mathbb Z/ \ell^{n_\ell} \mathbb Z)$ satisfying Condition \ref{hensel condition}. Using these, compute matrices corresponding to the $s^1,s^2,s^3,s^4$, which in any case are well defined modulo $\ell^{2m_\ell+2}$. Denote these by $S_\ell^i \in \mathrm{M}_2(\mathbb Z/ \ell^{2m_\ell + 2} \mathbb Z)$ for each $\ell \in \mathcal V$. By expressing each $w_\ell^j \in \mathcal B_j$ as a $\mathbb Z$-linear combination of $s^1,s^2,s^3,s^4$, we can compute matrices $W_\ell ^j \in  \mathrm{M}_2(\mathbb Z/\ell^{2m_\ell +2}\mathbb Z )$ corresponding to the $w_\ell ^j$.

\begin{rem}
So far we have only mentioned $\ell_0  $ as that it needed to be added to $\mathcal V$. If one wanted to compute several Hecke operators at once, they could add all the primes for the operators into $\mathcal V$. Then in the above steps we computed all the relevant matrices, to save us repeating the calculations if we were to compute the Hecke operators one by one. 
\end{rem}

\item\label{step2} Let $\mathbbm{1}_{w^1},\dots, \mathbbm{1}_{w^h}$ be the characteristic functions of the points $w^1,\dots,w^h \in \Omega$. This is a basis for the vector space of $\overline{\mathbb F}_p$-valued functions on $\Omega$. Let $g_0, \dots, g_{\ell_0  -1},g_{\ell_0  } \in \mathrm{M}_2(\mathbb Z_{\ell_0  })$ be the matrices \\$\begin{psmallmatrix} 1 & 0 \\ 0 & \ell_0   \end{psmallmatrix}, \begin{psmallmatrix} 1 & 1 \\ 0 & \ell_0   \end{psmallmatrix}, \dots, \begin{psmallmatrix} 1 & \ell_0  -1 \\ 0 & \ell_0   \end{psmallmatrix}, \begin{psmallmatrix} \ell_0   & 0 \\ 0 & 1 \end{psmallmatrix}$. Define the quantity $e_{i,j,k}$ for $1 \leq i,j \leq h$ and $0 \leq k \leq \ell_0  $ by
$$e_{i,j,k} = \mathbbm{1}_{w^i}(g_k \cdot w^j).$$ Then by definition we have the formula
$$T_{\ell_0  }(\mathbbm{1}_{w^i})(w^j) = \frac{1}{\ell_0  } \cdot \sum\limits_{k=0}^{\ell_0  } e_{i,j,k}.$$
This is then the $(j,i)$-th entry of the matrix for $T_{\ell_0  }$ with respect to the basis $\{\mathbbm{1}_{w^1},\dots,\mathbbm{1}_{w^h}\}$. Hence it remains to compute this quantity $e_{i,j,k}$.

\item Fix $i,j$. Determine the set $\mathcal V_{i,j}$. Compute the quantities
$$M = \ell_0  \prod\limits_{\ell \in \mathcal V_{i,j} \cup \{p,\ell_0\}} \ell^{m_\ell }$$ and $$K = \ell_0  \prod\limits_{\ell \in \mathcal V_{i,j} \cup \{p,\ell_0\}} \ell^{v_\ell (\mathrm{nrd}(w_\ell ^j)) - v_\ell (\mathrm{nrd}(w_\ell ^i))}.$$

\item Check if there exist integers $t,x,y,z \in \mathbb Z$ such that the following conditions hold:
\begin{equation*}
\begin{dcases*}
\mathrm{nrd}(t\cdot s^1+x\cdot s^2 + y\cdot s^3 + z \cdot s^4) = KM^2 \\
W_\ell ^i \cdot \begin{pmatrix} a_\ell & b_\ell  \\ c_\ell  &-a_\ell \end{pmatrix} \cdot\mathrm{adj}(W_\ell ^j) \in \ell^{m_\ell  + v_\ell (\det(W_\ell ^j))} \cdot \mathrm{M}_2(\mathbb Z/ \ell^{2m_\ell +2} \mathbb Z) \hspace{3em} \text{ for $\ell \in \mathcal V_{i,j}$} \\
W_{\ell_0  }^i \cdot \begin{pmatrix} a_{\ell_0  }& b_{\ell_0  } \\ c_{\ell_0  } &-a_{\ell_0  }\end{pmatrix} \cdot \mathrm{adj}(W_{\ell_0  }^j) \cdot \mathrm{adj}(g_k)
 \in \ell_0  ^{m_{\ell_0  } + v_{\ell_0  }(\det(W_{\ell_0  }^j))+2} \cdot \mathrm{M}_2(\mathbb Z/ \ell_0  ^{2m_{\ell_0} +2} \mathbb Z).
\end{dcases*}
\end{equation*}
where $\begin{psmallmatrix} a_\ell & b_\ell \\ c_\ell & -a_\ell \end{psmallmatrix} \in \mathrm{M}_2(\mathbb Z/ \ell  ^{2m_\ell +2} \mathbb Z)$ is the matrix $t\cdot S_\ell^1+x\cdot S_\ell^2 + y\cdot S_\ell^3 + z \cdot S_\ell^4$.
If such $t,x,y,z$ exist, set $e_{i,j,k}$ to be 1, and otherwise 0. Note that the only dependence on $k$ is in the last condition.

\end{enumerate}
\end{algo}

\section{Introducing weight and level}

We now discuss what modifications must be done when computing the Hecke eigenvalues mod $p$ of general weight and level on the quaternion side. In the previous section we checked whether two elements of $\Omega$ are the same by checking if they determine isomorphic invertible left $\mathcal O$-ideals. The point is that if we now attach $\pi N$-structure to our ideals, we need to determine whether this isomorphism sends one $\pi N$-structure to the other. Our aim is to compute the Hecke operators $T_{\ell_0  }$, for $\ell_0  \nmid pN$, on the space of functions $\Omega(N) = U(N) \backslash D^\times (\mathbb A_f)/D^\times (\mathbb Q) \to \overline{\mathbb F}_p$. Recall this was given by $$T_{\ell_0  } f([x_\ell ]) = \ell_0  ^{-1}\sum\limits_i f(g_i \cdot [x_{\ell  }])$$for $\mathrm{GL}_2(\mathbb Z_{\ell_0  }) \begin{psmallmatrix} 1 & 0 \\ 0 & \ell_0   \end{psmallmatrix} \mathrm{GL}_2(\mathbb Z_{\ell_0  }) = \bigsqcup \mathrm{GL}_2(\mathbb Z_{\ell_0  })g_i$. If we wanted to isolate the Hecke eigenvalues arising from weight $k \mod p^2-1$, then we need to restrict to the functions which satisfy $f(\mu \cdot [x_\ell ]) = \mu^{-k}f([x_\ell ])$, where $\mu \in \mathcal O_p^\times/ \mathcal O_p^\times(1) \cong \mathbb F_{p^2}^\times$ acts on $[x_\ell ]$ by multiplication in the $p$-place. In our case, we can identify $\mathcal O_p^\times/ \mathcal O_p^\times(1) \cong \mathbb F_{p^2}^\times$ with $\{s+ti \mid s,t \in \mathbb F_{p} \text{ not both zero}\}$, which is closed under multiplication, where $i$ and $j$ are the generators of $D$ ($j$ can be viewed as a uniformiser $\pi$ of $\mathcal O_p^\times$).

Firstly, we write down the elements of $\Omega(N)$. For $U$ defined as $\mathcal O_p^\times \times \prod\limits_{\ell \neq p} \mathcal O_\ell ^\times$ previously, we have $$U(N) \backslash U \cong \mathbb F_{p^2}^\times \times \prod\limits_{\ell \neq p} \mathrm{GL}_2(\mathbb Z/\ell^{v_\ell (N)}\mathbb Z) \cong \mathbb F_{p^2}^\times \times \mathrm{GL}_2(\mathbb Z/N\mathbb Z).$$
Hence, for our points $w^1,\dots, w^h$ of $\Omega$, we need to multiply on the left by representatives of $\mathbb F_{p^2}^\times$ and $\mathrm{GL}_2(\mathbb Z/N\mathbb Z)$ at the appropriate places, in order to determine all the elements of $\Omega(N)$. Ranging over each choice of $\mu \in \mathbb F_{p^2}^\times$ and $\gamma \in \mathrm{GL}_2(\mathbb Z/N\mathbb Z)$, the corresponding $(U(N),D^\times(\mathbb Q))$-double cosets cover $D^\times (\mathbb A_f)$, but need not be distinct. For example, for any triple $$\vec{j} = (j,\mu,\gamma) \in \{1,\dots,h\} \times \mathbb F_{p^2}^\times \times \mathrm{GL}_2(\mathbb Z/N\mathbb Z),$$ if we define $v^{\vec{j}}$ as in the Updated Table of Notation below, then $\vec{j} = (j,\mu,\gamma)$ and $\vec{j}' = (j,-\mu,-\gamma)$ define the same element $v^{\vec{j}} = v^{\vec{j}'}$ of $\Omega(N)$. This example comes from the fact that $-1 \in D^\times (\mathbb Q)$ and $-1 \in \mathcal O_\ell^\times$ (but $-1 \not\in \mathcal O_p^\times(1)$ and $-1 \not\in \mathcal O_\ell^\times(N)$ for $\ell \mid N$), so that
\begin{align*}
[w_2^j, w_3^j, \dots, \mu \cdot w_p^j, \dots, \gamma_\ell \cdot w_\ell^j , \dots] &= [-w_2^j, -w_3^j, \dots, -\mu \cdot w_p^j, \dots, -\gamma_\ell \cdot w_\ell^j , \dots] \\  &= [w_2^j, w_3^j, \dots, -\mu \cdot w_p^j, \dots, -\gamma_\ell \cdot w_\ell^j , \dots].
\end{align*}
When interpreting elements of $\Omega(N)$ in terms of isomorphism classes of invertible left $\mathcal O$-ideals with $\pi N$-structure, this phenomenon is due to automorphisms of the left $\mathcal O$-ideals, which then shift around the $\pi N$-structure; in our example we always have the automorphism given by multiplication by $-1$, which sends $\mu \mapsto -\mu$, $\gamma \mapsto -\gamma$. It is possible to determine these automorphisms. Recall that the points $w^1,\dots, w^h \in \Omega$ give local generators for representatives $I_1,\dots, I_h$ of the left ideal classes of $\mathcal O$. The automorphisms of $I_j$ as a left $\mathcal O$-ideal are precisely given by right multiplication by the units of the right order $O_R(I_j)$ of $I_j$, where we define this as in \cite{Voi1} to be:
$$O_R(I_j) := \{ \alpha \in D \mid I_j \alpha \subset I_j\}.$$
Note that $O_R(I_j)^\times$ will always be a finite set, once again because the quaternion algebra $D$ is ramified at infinity.

\begin{rem}
In Serre's letter \cite{Ser1}, the relationship between Hecke eigenvalues on the modular form and quaternion sides arises as a result of some generalisation of the Deuring correspondence. Classically, this establishes an equivalence of categories between supersingular elliptic curves mod $p$ under isogenies, and invertible left $\mathcal O$-ideals under nonzero left $\mathcal O$-module homomorphisms (for a maximal order $\mathcal O$ of the quaternion algebra $D$ ramified at $\{p,\infty\}$). See Theorem 42.3.2 of \cite{Voi1} for a reference. Then, if a supersingular elliptic curve $E$ corresponds to an invertible left $\mathcal O$-ideal $I$, we see that $O_R(I)^\times \cong \mathrm{Aut}(E)$. But the automorphism group of an elliptic curve is well understood. As in Theorem III.10.1 of \cite{Sil1}, the automorphism group of $E$ has order dividing 24, and if $j(E) \neq 0,1728$, then $\mathrm{Aut}(E)$ has order 2 and is just given by $\pm 1$.
\end{rem}

We have seen that in listing the elements of $\Omega(N)$, we need to identify, for example, $v^{\vec{j}}$ and $v^{\vec{j}'}$ when $\vec{j} = (j,\mu,\gamma)$ and $\vec{j}' = (j,-\mu,-\gamma)$, because $-1 \in O_R(I_j)^\times$. In general, if $\zeta \in O_R(I_j)^\times$, we know 
\begin{align*}
[w_2^j, w_3^j, \dots, \mu \cdot w_p^j, \dots, \gamma_\ell \cdot w_\ell^j , \dots] &= [w_2^j \zeta, w_3^j \zeta, \dots, \mu \cdot w_p^j \zeta , \dots, \gamma_\ell \cdot w_\ell^j \zeta, \dots] \\
 &= [w_2^j, w_3^j, \dots, \mu\phi_j(\zeta) \cdot w_p^j, \dots, \gamma_\ell \psi_j(\zeta;\ell) \cdot w_\ell^j , \dots]\end{align*}
for $\phi_j(\zeta)$ and $\psi_j(\zeta;\ell)$ as defined in the Updated Table of Notation. To define this, we use the fact that because $I_j \zeta = I_j$, we have for any $\ell$, $\mathcal O_\ell  w_\ell^j \zeta= \mathcal O_\ell w_\ell^j $, and so $w_\ell^j \zeta (w_\ell^j)^{-1} \in \mathcal O_\ell^\times$. Let $\psi_j(\zeta) \in \mathrm{GL}_2(\mathbb Z/N\mathbb Z)$ be the matrix congruent to $\psi_j(\zeta;\ell)$ mod $\ell^{v_\ell(N)}$ for each $\ell \mid N$.

From this, we are motivated to define the set $$\mathcal S = \left(\{1,\dots,h\} \times \mathbb F_{p^2}^\times \times \mathrm{GL}_2(\mathbb Z/N\mathbb Z) \right) / \sim,$$ where we identify ${(j,\mu,\gamma) \sim (j,\mu  \phi_j(\zeta), \gamma \psi_j(\zeta))}$ for any $\zeta \in O_R(I_j)^\times$. Then we can enumerate the elements of $\Omega(N)$ as $\{v^{\vec{j}} \mid \vec{j} \in \mathcal S\}$. Note that $v^{\vec{j}}$ reduces to $w^j$ when we quotient by $U$ in $\Omega = U \backslash \Omega(N)$. Once again, an obvious choice of basis for functions $\Omega(N) \to \overline{\mathbb F}_p$ is given by the characteristic functions $\mathbbm{1}_{v^{\vec{j}}}$. One is then led to computing the following quantity: $$e_{\vec{i},\vec{j},k} = \mathbbm{1}_{v^{\vec{i}} }(g_k \cdot v^{\vec{j}} ) \in \{0,1\}$$ for $\vec{i},\vec{j} \in \mathcal S$. Then we have the formula 
$$T_{\ell_0  }(\mathbbm{1}_{v^{\vec{i}}})(v^{\vec{j}}) = \frac{1}{\ell_0  } \cdot \sum\limits_{k=0}^{\ell_0  } e_{\vec{i},\vec{j},k}.$$
If we index the rows and columns of the matrix for $T_{\ell_0  }$, with respect to this basis, by $\vec{i} \in \mathcal S$, then the $(\vec{j},\vec{i})$-th entry is this value above. We observe that $e_{\vec{i},\vec{j},k}=0$ if $e_{i,j,k}=0$, for $e_{i,j,k}$ as defined in step \ref{step2} in Algorithm \ref{algo}. This is because if $v^{\vec{i}} = g_k \cdot v^{\vec{j}} \in \Omega(N)$, then they generate left $\mathcal O$-ideals with $\pi N$-structure that are isomorphic; in particular the ideals are isomorphic, and so $w^i = g_k \cdot w^j$. One can think of this as replacing $e_{i,j,k}$ in Algorithm \ref{algo} with a permutation matrix, depending on $i,j,k$, keeping track of the $\pi N$-structure.

So now we see how to compute $e_{\vec{i},\vec{j},k}$. It is 0 if $e_{i,j,k}=0$. If $e_{i,j,k}=1$, then from Algorithm \ref{algo} we compute some $\alpha \in D^\times (\mathbb Q)$ such that for all $\ell \neq \ell_0  $,
\begin{equation*}w_\ell ^i \alpha (w_\ell ^j)^{-1} \in \mathcal O_\ell^\times,\end{equation*}
and also
\begin{equation*}w_{\ell_0  }^i \alpha (w_{\ell_0  }^j)^{-1} g_k^{-1} \in \mathcal O_{\ell_0}^\times.\end{equation*}
Pick representatives $(i,\mu,\gamma)$ and $(j,\mu',\gamma')$ for $\vec{i}, \vec{j}$. What we need to check now is whether we also have 
\begin{equation}\label{maineq3}
(\mu \phi_i(\zeta_i)) \cdot w_p^i \cdot \alpha \cdot (w_p^j)^{-1} \cdot (\mu' \phi_j(\zeta_j))^{-1} \in \mathcal O_p^\times(1)
\end{equation}
and for $\ell \mid N$
\begin{equation}\label{maineq4}
(\gamma_\ell \psi_i(\zeta_i;\ell)) \cdot w_\ell ^i \cdot \alpha \cdot (w_\ell ^j)^{-1} \cdot (\gamma'_\ell \psi_j(\zeta_j ; \ell))^{-1} \in \mathcal O_\ell ^\times (N)
\end{equation}
for some $\zeta_i \in O_R(I_i)^\times$ and $\zeta_j \in O_R(I_j)^\times$. In other words, we need to compute $w_p^i \cdot \alpha \cdot (w_p^j)^{-1}$, which we know is in $\mathcal O_p^\times$ by definition of $\alpha$, and then check if this reduces to $(\mu \phi_i(\zeta_i))^{-1}(\mu' \phi_j(\zeta_j)) \in \mathbb F_{p^2}^\times$ modulo the uniformiser $\pi$, for some $\zeta_i \in O_R(I_i)^\times$ and $\zeta_j \in O_R(I_j)^\times$. We also need to compute for each $\ell \mid N$ the terms $w_\ell ^i \cdot \alpha \cdot (w_\ell ^j)^{-1}$, which we know are in $\mathcal O_\ell ^\times$, and then check whether they reduce to $(\gamma_\ell \psi_i(\zeta_i;\ell))^{-1}(\gamma'_\ell \psi_j(\zeta_j ; \ell)) \in \mathrm{GL}_2(\mathbb Z/\ell^{v_\ell (N)} \mathbb Z)$ modulo $\ell^{v_\ell (N)}$, for the same $\zeta_i \in O_R(I_i)^\times$ and $\zeta_j \in O_R(I_j)^\times$. To make sense of this, we need to write down matrices at the primes dividing $N$ as well, so we extend our set $\mathcal V$ in the Updated Table of Notation (changes from the previous table in bold).

We are now ready to write out the algorithm for the general case.

\newpage

\begin{center}
\textbf{Updated Table of Notation}
\end{center}
\begin{center}

\binoppenalty=9999
\relpenalty = 9999
	\begin{tabular}{|V{18em} |V{20em}|V{9.5em}|}
	\hline
\textbf{Notation} & \textbf{Definition} & \textbf{Corresponding Matrices} \\ \hline

$D = \left(\frac{-\epsilon,-p}{\mathbb Q}\right)$ & Quaternion algebra ramified at $\{p,\infty\}$. & \\ \hline
$\mathcal O$ & A maximal order of $D$. & \\ \hline
$I_1,\dots, I_h$ & Representatives of the left ideal classes of $\mathcal O$. & \\ \hline
$\mathcal B_1,\dots, \mathcal B_h$ & $\mathbb Z$-bases for $I_1,\dots,I_h$. & \\ \hline
$O_R(I_1)^\times, \dots, O_R(I_h)^\times$ & The units of the right orders of $I_1,\dots, I_h$. & \\ \hline
$\mathcal V$ & The set of primes $\ell$ for which in some $\mathcal B_j$ all elements have reduced norm divisible by $\ell$,  excluding $p$, and including $\ell_0  $ \textbf{and all primes dividing $N$}. These are the primes at which we need to compute matrices. & \\ \hline
$i,j$ & & $A_\ell ,B_\ell  \in \mathrm{M}_2(\mathbb Z/\ell^{n_\ell }\mathbb Z)$ for each $\ell \in \mathcal V$.  \\ \hline
$\{s^1,s^2,s^3,s^4\}$ & & $S^1_\ell ,S_\ell ^2,S_\ell ^3,S_\ell ^4 \in \mathrm{M}_2(\mathbb Z/\ell^{2m_\ell \mathbf{+v_\ell (N)}+2} \mathbb Z)$ for each $\ell \in \mathcal V$.  \\ \hline
$w^j=[w_2^j,w_3^j,w_5^j, \dots] \in \Omega$ & The elements of $\Omega$ corresponding to the $I_j$. The square brackets means the double coset represented by the adelic point $(w_2^j,w_3^j,\dots)$. From the way we compute this, almost all $w_\ell ^j$ will be 1.  & $W_\ell ^j \in \mathrm{M}_2(\mathbb Z/ \ell^{2m_\ell \mathbf{+v_\ell (N)}+2} \mathbb Z)$ corresponding to the $w_\ell ^j$ for each $\ell \in \mathcal V$.  \\ \hline
$v_\ell (N)$ & The $\ell$-adic valuation of $N$. & \\ \hline
$m_\ell $, $n_\ell =2m_\ell \mathbf{+ v_\ell (N)}+4$ for $\ell \neq 2$, and $n_2 = \max(7,2m_2 +\mathbf{v_2(N)}+ 4)$ &  $m_\ell := \max\limits_j(v_\ell(\mathrm{nrd}(w_\ell^j)))$. & \\ \hline
$g_0,\dots,g_{\ell_0  } \in \mathrm{M}_2(\mathbb Z_{\ell_0  })$ & The matrices  $\begin{psmallmatrix} 1 & 0 \\ 0 & \ell_0   \end{psmallmatrix}, \begin{psmallmatrix} 1 & 1 \\ 0 & \ell_0   \end{psmallmatrix}, \dots, $ $\begin{psmallmatrix} 1 & \ell_0  -1 \\ 0 & \ell_0   \end{psmallmatrix}, \begin{psmallmatrix} \ell_0   & 0 \\ 0 & 1 \end{psmallmatrix}$. & \\ \hline
$\mathcal O_p^\times / \mathcal O_p^\times (1) \cong \mathbb F_{p^2}^\times$ & Identified with ${\{s+ti \mid s,t \in \mathbb F_{p} \text{ not both zero}\}}$. & \\ \hline

	\end{tabular}
\end{center}

\newpage

\begin{center}

\binoppenalty=9999
\relpenalty = 9999
	\begin{tabular}{|V{18em} |V{20em}|V{9.5em}|}
	\hline
\textbf{Notation} & \textbf{Definition} & \textbf{Corresponding Matrices} \\ \hline
$\phi_j(\zeta) \in \mathbb F_{p^2}^\times$ & For $\zeta \in O_R(I_j)^\times$, let $\phi_j(\zeta)$ be the image of ${w_p^j \cdot \zeta \cdot (w_p^j)^{-1} \in \mathcal O_p^\times}$ in $\mathbb F_{p^2}^\times$. & \\ \hline
$\psi_j(\zeta;\ell) \in \mathrm{GL}_2(\mathbb Z/\ell^{v_\ell(N)}\mathbb Z)$ for $\ell \mid N$, \newline $\psi_j(\zeta) \in \mathrm{GL}_2(\mathbb Z/N\mathbb Z)$ & For $\zeta \in O_R(I_j)^\times$, consider ${w_\ell^j \cdot \zeta \cdot (w_\ell^j)^{-1} \in \mathcal O_\ell^\times \cap D}$ for all $\ell \mid N$. This lives in $\mathcal O$, so we can compute the corresponding matrix in $\mathrm{GL}_2(\mathbb Z/\ell^{n_\ell}\mathbb Z)$, and consider the reduction $\psi_j(\zeta;\ell) \in \mathrm{GL}_2(\mathbb Z/\ell^{v_\ell(N)}\mathbb Z)$. Let $\psi_j(\zeta) \in \mathrm{GL}_2(\mathbb Z/N\mathbb Z)$ be the matrix congruent to $\psi_j(\zeta;\ell)$ mod $\ell^{v_\ell(N)}$ for each $\ell \mid N$.  & \\ \hline
$\mathcal S = \left(\{1,\dots,h\} \times \mathbb F_{p^2}^\times \times \mathrm{GL}_2(\mathbb Z/N\mathbb Z) \right) / \sim$ & Here we identify ${(j,\mu,\gamma) \sim (j,\mu  \phi_j(\zeta), \gamma \psi_j(\zeta))}$ for any $\zeta \in O_R(I_j)^\times$.        &  \\ \hline

$v^{\vec{j}} = [v_2^{\vec{j}}, v_3^{\vec{j}}, \dots] \in \Omega(N)$ for $\vec{j} = [(j,\mu,\gamma)] \in \mathcal S$  & Here \begin{equation*}
v_\ell^{\vec{j}} =
\begin{dcases*}
w_\ell ^j & if $\ell \neq p \text{ and } \ell \nmid N$ \\
\mu \cdot w_p^j & if $ \ell=p$ \\
\gamma_\ell  \cdot w_\ell ^j & if $\ell \mid N$
\end{dcases*}
\end{equation*} for $\mu \in \mathbb F_{p^2}^\times $ and $\gamma \in \mathrm{GL}_2(\mathbb Z/ N \mathbb Z)$ with reduction $\gamma_\ell $ mod $\ell^{v_\ell (N)}$. To be precise, we should choose lifts of $\mu$ to $\mathcal O_p^\times$ and $\gamma_\ell $ to $\mathcal O_\ell ^\times \cong \mathrm{GL}_2(\mathbb Z_\ell )$. By construction, $v^{\vec{j}} \in \Omega(N)$ is well defined for any choice of representative $(j,\mu,\gamma)$ of $\vec{j} \in \mathcal S$.  & \\ \hline
$\mathbbm{1}_{v^{\vec{j}} }$ & The characteristic function $\Omega(N) \to \overline{\mathbb F}_p$ of the point $v^{\vec{j}}  \in \Omega(N)$. & \\ \hline
$e_{\vec{i},\vec{j},k} \in \{0,1\}$ & $e_{\vec{i},\vec{j},k} = \mathbbm{1}_{v^{\vec{i}} }(g_k \cdot v^{\vec{j}} )$ for $\vec{i},\vec{j} \in \mathcal S$.  & \\ \hline
$\mathcal V_{i,j} \subset \mathcal V$ & The set of primes $\ell$ such that at least one of $w_\ell ^i$ and $w_\ell ^j$ is not 1, excluding $p$ and $\ell_0  $, \textbf{and including all primes dividing $N$}. & \\ \hline
	\end{tabular}
\end{center}

\newpage

\begin{algo}\label{main algo}
Input: a prime $p$ and level $N$ coprime to $p$, together with a prime $\ell_0  $ coprime to $pN$.

Output: a matrix representing the action of the Hecke operator $T_{\ell_0  }$ on the space of all functions $\Omega(N) \to \overline{\mathbb F}_p$.
\begin{enumerate}[label=(\roman*)]
\item Define a quaternion algebra $D = \left(\frac{-\epsilon,-p}{\mathbb Q}\right)$ over $\mathbb Q$, ramified exactly at $\{p,\infty\}$. Define a maximal order $\mathcal O$ with integer basis given as in Proposition \ref{pizer basis}, for which we denote the basis elements $\{s^1,s^2,s^3,s^4\}$. Compute the left ideal classes $I_1, \dots, I_h$ of $\mathcal O$, and bases $\mathcal B_1, \dots, \mathcal B_h$ for them.

\item Compute the points $w^j = [w_2^j, w_3^j, w_5^j, \dots] \in \Omega$ corresponding to $I_j$ for each $j$ as follows. We take $w_\ell ^j$ to be any generator of $I_j \otimes \mathbb Z_\ell $ (which we know is principal). To do this, for our basis $\mathcal B_j$, compute the reduced norm of each of the four elements and set $w_\ell ^j$ to be any of these elements whose reduced norm has minimal $\ell$-adic valuation. Note that for almost all $\ell$ this valuation is zero, so we can instead take $w_\ell ^j=1$, and do so when possible.

\item Determine the set $\mathcal V$, defined in the Updated Table of Notation. For each $\ell \in \mathcal V \cup \{p\}$, compute $m_\ell$ and $n_\ell$. For each $\ell \in \mathcal V$, compute matrices $A_\ell, B_\ell \in \mathrm{M}_2(\mathbb Z/ \ell^{n_\ell} \mathbb Z)$ satisfying Condition \ref{hensel condition}. Using these, compute matrices corresponding to the $s^1,s^2,s^3,s^4$, which in any case are well defined modulo $\ell^{2m_\ell+v_\ell(N)+2}$. Denote these by $S_\ell^i \in \mathrm{M}_2(\mathbb Z/ \ell^{2m_\ell +v_\ell(N)+ 2} \mathbb Z)$ for each $\ell \in \mathcal V$. By expressing each $w_\ell^j \in \mathcal B_j$ as a $\mathbb Z$-linear combination of $s^1,s^2,s^3,s^4$, we can compute matrices $W_\ell ^j \in  \mathrm{M}_2(\mathbb Z/\ell^{2m_\ell +v_\ell(N)+2}\mathbb Z )$ corresponding to the $w_\ell ^j$.

\item Compute each $O_R(I_j)^\times$, and for each $\zeta \in O_R(I_j)^\times$, compute $\phi_j(\zeta) \in \mathbb F_{p^2}^\times$ and $\psi_j(\zeta;\ell) \in \mathrm{GL}_2(\mathbb Z/ \ell^{v_\ell(N)}\mathbb Z)$ for $\ell \mid N$. Determine representatives for $\mathcal S$. Fix $i,j$. We will compute $e_{\vec{i},\vec{j},k}$ for $0 \leq k \leq \ell$ and $\vec{i}=(i,\cdot,\cdot), \hspace{0.5em}\vec{j} = (j,\cdot,\cdot) \in \mathcal S$ among these representatives. Determine the set $\mathcal V_{i,j}$. Compute $$M = \ell_0   \prod\limits_{\ell \in \mathcal V_{i,j} \cup \{p,\ell_0\}} \ell^{m_\ell }$$ and $$K = \ell_0  \prod\limits_{\ell \in \mathcal V_{i,j} \cup \{p,\ell_0\}} \ell^{v_\ell (\mathrm{nrd}(w_\ell ^i)) - v_\ell (\mathrm{nrd}(w_\ell ^j))}.$$

\item We firstly compute $e_{i,j,k}$ as in Algorithm \ref{algo}. Check if there exist integers $t,x,y,z \in \mathbb Z$ such that the following conditions hold:
\begin{equation*}
\begin{dcases*}
\mathrm{nrd}(t\cdot s^1+x\cdot s^2 + y\cdot s^3 + z \cdot s^4) = KM^2 \\
W_\ell ^i \cdot \begin{pmatrix} a_\ell & b_\ell  \\ c_\ell  &-a_\ell \end{pmatrix} \cdot\mathrm{adj}(W_\ell ^j) \in \ell^{m_\ell  + v_\ell (\det(W_\ell ^j))} \cdot \mathrm{M}_2(\mathbb Z/ \ell^{2m_\ell +v_\ell(N)+2} \mathbb Z) \hspace{3em} \text{ for $\ell \in \mathcal V_{i,j}$} \\
W_{\ell_0  }^i \cdot \begin{pmatrix} a_{\ell_0  }& b_{\ell_0  } \\ c_{\ell_0  } &-a_{\ell_0  }\end{pmatrix} \cdot \mathrm{adj}(W_{\ell_0  }^j) \cdot \mathrm{adj}(g_k)
 \in \ell_0  ^{m_{\ell_0  } + v_{\ell_0  }(\det(W_{\ell_0  }^j))+2} \cdot \mathrm{M}_2(\mathbb Z/ \ell_0  ^{2m_{\ell_0} +2} \mathbb Z).
\end{dcases*}
\end{equation*}
where $\begin{psmallmatrix} a_\ell & b_\ell \\ c_\ell & -a_\ell \end{psmallmatrix} \in \mathrm{M}_2(\mathbb Z/ \ell  ^{2m_\ell +v_\ell(N)+2} \mathbb Z)$ is the matrix $t\cdot S_\ell^1+x\cdot S_\ell^2 + y\cdot S_\ell^3 + z \cdot S_\ell^4$.
If such $t,x,y,z$ exist, set $e_{i,j,k}$ to be 1, and otherwise 0. Note that the only dependence on $k$ is in the last condition.

\item If $e_{i,j,k}=0$, set $e_{\vec{i},\vec{j},k}=0$ for all $\vec{i}=(i,\cdot,\cdot)$, $\vec{j} = (j,\cdot,\cdot)$. Otherwise, taking our solution $(t,x,y,z)$ from the above step, compute the matrices
$$\overline{Q_\ell } := W_\ell ^i \cdot \frac{1}{M}\begin{pmatrix} a_\ell & b_\ell  \\ c_\ell  &-a_\ell \end{pmatrix} \cdot (W_\ell ^j)^{-1} \mod \ell^{v_\ell (N)}\in  \mathrm{GL}_2(\mathbb Z/ \ell^{v_\ell (N)} \mathbb Z)$$
for $\ell \mid N$, and 
$$Q_p := w_p^i \cdot \frac{1}{M}(t\cdot s^1+x\cdot s^2 + y \cdot s^3 + z \cdot s^4) \cdot (w_p^j)^{-1} \in \mathcal O_p^\times.$$
Writing $Q_p$ in terms of the generators $i$ and $j$ of the quaternion algebra $D$, let $\overline{Q_p} \in \mathbb F_{p^2}^\times$ be the reduction modulo $j$, where we view the elements of $\mathbb F_{p^2}^\times$ as $\{s+ti \mid s,t \in \mathbb F_p \text{ not both zero}\}$, which is a group under multiplication.

Then, for representatives $\vec{i}=(i,\mu,\gamma)$ and $\vec{j}=(j,\mu',\gamma')$ of $\mathcal S$, define

\begin{equation}
e_{\vec{i},\vec{j},k}=
\begin{dcases*}
1 & \parbox{38em}{if there exists $\zeta_i \in O_R(I_i)^\times$ and $\zeta_j \in O_R(I_j)^\times$ such that $(\mu \phi_i(\zeta_i))^{-1}(\mu' \phi_j(\zeta_j)) = \overline{Q_p} \in \mathbb F_{p^2}^\times$ and $(\gamma_\ell \psi_i(\zeta_i;\ell))^{-1}(\gamma'_\ell \psi_j(\zeta_j ; \ell)) = \overline{Q_\ell } \in \mathrm{GL}_2(\mathbb Z/ \ell^{v_\ell (N)}\mathbb Z)$ for all $\ell \mid N$} \\
0 & otherwise
\end{dcases*}
\end{equation}

\item If we index the rows and columns of the matrix for $T_{\ell_0  }$, with respect to the basis consisting of $\mathbbm{1}_{v^{\vec{i}}}$, by $\vec{i}\in \mathcal S$, then the $(\vec{j},\vec{i})$-th entry is $$ \frac{1}{\ell_0  } \cdot \sum\limits_{k=0}^{\ell_0  } e_{\vec{i},\vec{j},k}.$$ This gives us the matrix for $T_{\ell_0  }$.

\end{enumerate}
\end{algo}

\newpage
\bibliographystyle{amsalpha}
\bibliography{mybib}

\end{document}